\documentclass[a4paper, 10pt]{amsart}

\usepackage{amsmath,amssymb, color}

\theoremstyle{plain}
\newtheorem{theorem}{\bf Theorem}[section]
\newtheorem{proposition}[theorem]{\bf Proposition}
\newtheorem{lemma}[theorem]{\bf Lemma}
\newtheorem{corollary}[theorem]{\bf Corollary}

\theoremstyle{definition}

\numberwithin{equation}{section}

 \DeclareMathOperator{\ord}{ord}

 \DeclareMathOperator{\supp}{supp}

\renewcommand{\t}{\, | \,}

\begin{document}

\title{The catenary degree of  Krull monoids {II}}

\address{University of Graz, NAWI Graz, \\
Institute for Mathematics and Scientific Computing \\
Heinrichstra{\ss}e 36\\
8010 Graz, Austria}

\email{alfred.geroldinger@uni-graz.at, qinghai.zhong@uni-graz.at}

\author{Alfred Geroldinger  and Qinghai Zhong}

\thanks{This work was supported by
the Austrian Science Fund FWF, Project Number M1641-N26.}



\begin{abstract}
Let $H$ be a Krull monoid with finite class group $G$ such that
every class contains a prime divisor (for example, a ring of
integers in an algebraic number field or a holomorphy ring in an
algebraic function field). The catenary degree $\mathsf c (H)$ of
$H$ is the smallest integer $N$ with the following property: for
each $a \in H$ and each two factorizations $z, z'$ of $a$, there
exist factorizations $z = z_0, \ldots, z_k = z'$ of $a$ such that,
for each $i \in [1, k]$, $z_i$ arises from $z_{i-1}$ by replacing at
most $N$ atoms from $z_{i-1}$ by at most $N$ new atoms. To exclude trivial cases, suppose that $|G| \ge 3$. Then the catenary degree depends only on the class group $G$ and we have  $\mathsf c (H) \in [3, \mathsf D (G)]$, where $\mathsf D (G)$ denotes the Davenport constant of $G$. The cases when $\mathsf c (H) \in \{3,4,\mathsf D (G)\}$ have been previously characterized (see Theorem A). Based on a characterization of the catenary degree determined in the first paper \cite{Ge-Gr-Sc11a}, we determine the class groups satisfying  $\mathsf c (H)= \mathsf D (G)-1$. Apart from the mentioned extremal cases the precise value of $\mathsf c (H)$ is known for no further class groups.
\end{abstract}

\maketitle

\bigskip
\section{Introduction and Main Results}
\bigskip

As the title indicates, we continue the investigation of the arithmetic of Krull monoids. All integrally closed noetherian domains are Krull, and  holomorphy rings in global fields are Krull monoids  with finite class group and infinitely many prime ideals in each class.  A  Krull monoid is factorial if and only if its class group is trivial, and if this is not the case, then its arithmetic  is described by  invariants, such as sets of lengths and catenary degrees. We recall some basic definitions.

Let $H$ be a Krull monoid with class group $G$. Then each non-unit $a \in H$ can be written as a product of atoms, and if $a = u_1 \cdot \ldots \cdot u_k$ with atoms $u_1, \ldots, u_k$ of $H$, then $k$ is called the length of the factorization. The set of lengths $\mathsf L (a)$ of all possible factorization lengths is finite, and if $|\mathsf L (a)|>1$, then $|\mathsf L (a^n)| > n$ for each $n \in \mathbb N$. We denote by $\mathcal L (H) = \{ \mathsf L (a) \mid a \in H \}$ the system of sets of lengths of $H$. This is an infinite family of finite subsets of non-negative integers which is described by a variety  of arithmetical parameters. The present paper will focus on the three closely related invariants, namely the set of distances,  the $\daleth (H)$ invariant, and the catenary degree.
If  $L  = \{m_1, \ldots, m_l\} \subset \mathbb Z$ is a finite set of integers with $l \in \mathbb N$ and $m_1 < \ldots < m_l$, then $\Delta (L) = \{m_i - m_{i-1} \mid i \in [2,l]\} \subset \mathbb N$ is the set of distances of $L$. The set of distances $\Delta (H)$ of $H$ is the union of all sets $\Delta (L)$ with $L \in \mathcal L (H)$. If
\[
\daleth (H) = \sup \{ \min (L \setminus \{2\}) \mid 2 \in L \in \mathcal L (H) \} \,,
\]
then $\daleth (H) \le 2 + \sup \Delta (H)$.  The catenary degree
$\mathsf c (H)$ of $H$ is defined as the smallest integer $N$ with
the following property: for each $a \in H$ and each two
factorizations $z$ and $z'$ of $a$, there exist factorizations $z = z_0,
\ldots, z_k = z'$ of $a$ such that, for each $i \in [1, k]$, $z_i$
arises from $z_{i-1}$ by replacing at most $N$ atoms from $z_{i-1}$
by at most $N$ new atoms. A simple argument shows that $H$ is factorial if and only if $\mathsf c (H)=0$, and if this is not the case, then $2 + \sup \Delta (H) \le \mathsf c (H)$.

The study of these arithmetical invariants (in settings ranging from numerical monoids to Mori rings with zero-divisors) has attracted a lot of attention in the recent literature (for a sample see \cite{C-G-L-P-R06, C-G-L09, Om12a,Ga-Oj-SN13a, Ph14a,Ch-Sm13a,Ch-Go-Pe15a, C-C-M-M-P15}). Our main focus here will be on Krull monoids with finite class group $G$ such that each class contains a prime divisor. Let $H$ be such a Krull monoid. Then $|L|=1$ for all $L \in \mathcal L (H)$ if and only if $|G| \le 2$. Suppose that $|G| \ge 3$. Then the Davenport constant $\mathsf D (G)$ is finite, $\daleth (H) \ge 3$, and there is a  canonical chain of inequalities
\[
\daleth (H) \le 2 + \max \Delta (H) \le \mathsf c (H) \le \mathsf D (G) \,. \tag{$*$}
\]
In general, each inequality can be strict (see \cite[page 146]{Ge-Gr-Sc11a}). However, for the Krull monoids under consideration
the main result in \cite{Ge-Gr-Sc11a} states that $\daleth (H)=\mathsf c(H)$ holds under a certain mild assumption on the Davenport constant. Our starting point is the following Theorem {\bf A} (the first statement follows from \cite[Theorem 6.4.7]{Ge-HK06a}, and the characterization of $\mathsf c (H) \in [3,4]$ is given in \cite[Corollary 5.6]{Ge-Gr-Sc11a}).

\medskip
\noindent
{\bf Theorem A.} {\it Let $H$ be a Krull monoid with finite class group $G$ where $|G| \ge 3$ and each class contains a prime divisor. Then $\mathsf c (H) \in [3, \mathsf D (G)]$, and we have
\begin{enumerate}
\item $\mathsf c (H) = \mathsf D (G)$ if and only if $G$ is either cyclic or an elementary $2$-group.

\smallskip
\item $\mathsf c (H) = 3$ if and only if $G$ is isomorphic to one of the following groups\,{\rm :} $C_3, C_2 \oplus
      C_2, \mbox{ or } C_3 \oplus C_3$.

\smallskip
\item  $\mathsf c (H) = 4$ \ if and only if \
       $G$ is isomorphic to one of the following groups\,{\rm :} $ C_4, \,C_2 \oplus C_4,\, C_2 \oplus C_2 \oplus C_2,\,\mbox{ or } C_3
       \oplus C_3 \oplus C_3$.
\end{enumerate}}
\medskip

We formulate a main result of the present paper.

\medskip
\begin{theorem} \label{1.1}
Let $H$ be a Krull monoid with finite class group $G$ where $|G| \ge 3$ and each class contains a prime divisor. Then the following statements are equivalent{\rm \,:}
\begin{enumerate}
\item[(a)] $\mathsf c (H) = \mathsf D (G)-1$.

\smallskip
\item[(b)] $\daleth (H) = \mathsf D (G)-1$.

\smallskip
\item[(c)] $G$ is isomorphic either to $C_2^{r-1} \oplus C_4$ for some $r \ge 2$ or  to $C_2 \oplus C_{2n}$ for some $n \ge 2$.

\end{enumerate}
\end{theorem}

In order to discuss the statements of Theorem \ref{1.1} and their consequences,
let $H$ be a Krull monoid as in Theorem \ref{1.1}. Then the inequalities in ($*$) and the fact that
$\Delta (H)$ is an interval with $1 \in \Delta (H)$ (\cite{Ge-Yu12b}) imply that any of the following two conditions,
\[
\max \Delta (H)=\mathsf D (G)-3 \quad \text{or} \quad \Delta (H) = [1, \mathsf D (G)-3]
\]
is equivalent to the conditions in Theorem \ref{1.1}. The precise value of the Davenport constant is known for $p$-groups, for groups of rank at most two, and for some others. Thus we do know that $\mathsf D (C_2^{r-1}\oplus C_4)= r+3$ and that $\mathsf D (C_2 \oplus C_{2n})=2n+1$. But the value of $\mathsf D (G)$  is unknown  for general groups of rank three or for groups of the form $G=C_n^r$.  Even much less is known for the  catenary degree $\mathsf c (H)$ and for $\daleth (H)$. Their precise values are known only for the cases occurring in Theorem A and in Theorem \ref{1.1}.

\smallskip
As mentioned at the very beginning, holomorphy rings in global fields are (commutative) Krull monoids with finite class group. In recent years factorization theory has grown towards the non-commutative setting (e.g., \cite{B-P-A-A-H-K-M-R11, Ba-Ba-Go14}) with a focus on maximal orders in central simple algebras (they are non-commutative Krull monoids; see \cite{Sm13a, Ba-Sm15}). Combining these results with Theorem \ref{1.1} above, we obtain the following corollary.

\smallskip
\begin{corollary} \label{1.2}
Let $\mathcal O$ be a holomorphy ring in a global field $K$,  and $R$ a classical maximal $\mathcal O$-order in a central simple algebra $A$ over $K$ such that every stably free left $R$-ideal is free. Suppose that the  ray class group $G=\mathcal C_A ( \mathcal O)$ of $\mathcal O$ has at least three elements, let $\mathsf d$ denote a $($suitable$)$ distance on $R$, and $\mathsf c_{\mathsf d}(R)$ the $\mathsf d$-catenary degree of $R$. Then the following statements are equivalent{\rm \,:}
\begin{enumerate}
\item[(a)] $\mathsf c_{\mathsf d} (R) = \mathsf D (G)-1$.

\smallskip
\item[(b)] $\daleth (R) = \mathsf D (G)-1$.

\smallskip
\item[(c)] $G$ is isomorphic either to $C_2^{r-1} \oplus C_4$ for some $r \ge 2$ or  to $C_2 \oplus C_{2n}$ for some $n \ge 2$.

\end{enumerate}
\end{corollary}

\smallskip
Sets of lengths are the most investigated invariants in factorization theory. A standing (but wide open) conjecture states that for the class of Krull monoids under consideration sets of lengths are characteristic. To be more precise, let $H$ and $H'$ be Krull monoids with finite class groups $G$ and $G'$ with $|G|\ge |G'| \ge 4$ and suppose that each class contains a prime divisor. As usual, we write $\mathcal L (G)=\mathcal L (H)$ and $\mathcal L (H')=\mathcal L (G')$ (see Proposition \ref{2.1}). Then the conjecture states that $\mathcal L (G)=\mathcal L (G')$ implies that $G$ and $G'$ are isomorphic. For recent work in this direction we refer to \cite{Sc09c,Sc09b,B-G-G-P13a} or to \cite[Section 7.3]{Ge-HK06a}, \cite{Ge09a} for an overview. It turns out the extremal values of $\daleth (H)$ discussed in Theorem \ref{1.1} are the main tool to derive an arithmetical characterization of the associated groups. Thus we obtain the following corollary.

\medskip
\begin{corollary} \label{1.3}
Let $G$ be an abelian group.
\begin{enumerate}
\item If $\mathcal L (G) = \mathcal L ( C_2^{r-1} \oplus C_4 )$ for some $r \ge 2$, then $G \cong C_2^{r-1} \oplus C_4$.

\smallskip
\item If $\mathcal L (G) = \mathcal L ( C_2 \oplus C_{2n} )$ for some $n \ge 2$, then $G \cong C_2 \oplus C_{2n}$.
\end{enumerate}
\end{corollary}

\smallskip
In Section \ref{2} we gather together the required concepts and tools.  Section \ref{3} studies sets of lengths in monoids of zero-sum sequences over finite abelian groups, and it is mainly confronted with problems belonging to  structural (or inverse) additive number theory. Clearly, the irreducible elements of the monoids are precisely the minimal zero-sum sequences, and we mainly have to deal with minimal zero-sum sequences of extremal length $\mathsf D (G)$. The structure of minimal zero-sum sequences of length $\mathsf D (G)$ is known for cyclic groups and elementary $2$-groups (in both cases there are trivial answers), and for groups of rank two (\cite{Ga-Ge-Gr10a,Sc10b,Re15c}). Apart from that, structural results are available only in very special cases (\cite{Sc11b, Sa-Ch14a}), and this is precisely the lack of information which causes the difficulties in Section \ref{3} (see  Prop. \ref{3.5} - \ref{3.8}).
The proofs of the main results  are given in Section \ref{4}. They substantially use  transfer results (as partly summarized in Proposition \ref{2.1}, and also the transfer machinery  from \cite{Sm13a, Ba-Sm15}), the  work from \cite{Ge-Gr-Sc11a} (which relates the catenary degree and the $\daleth (\cdot)$ invariant of Krull monoids), and all the work from Section \ref{3}.

\medskip
\section{Preliminaries} \label{2}
\medskip

We denote by $\mathbb N$ the set of positive integers and by $\mathbb N_0$ the set of non-negative integers. For $n \in \mathbb N$, $C_n$ means a cyclic group of order $n$. For integers $a, b \in \mathbb Z$, $[a, b] = \{ x \in \mathbb Z \mid a \le x \le b\}$ is the discrete interval between $a$ and $b$. Let $A, B \subset \mathbb Z$ be subsets. Then $A+B=\{a+b \mid a \in A, b \in B\}$ denotes their sumset. If $A = \{ a_1, \ldots, a_l\}$ is finite with  $|A|=l\in \mathbb N_0$ and $a_1 < \ldots < a_l$, then $\Delta (A)=\{ a_i - a_{i -1} \mid i \in [2, l] \} \subset \mathbb N$ is the set of distances of $A$. By definition, $\Delta (A)=\emptyset$ if and only if $|A|\le 1$.

By a  monoid,  we always mean a commutative
semigroup with identity which satisfies the cancellation law (that
is, if $a,b ,c$ are elements of the monoid with $ab = ac$, then $b =
c$ follows). Let $H$ be a monoid. Then $H^{\times}$ denotes the unit group, $\mathsf q (H)$  the quotient group, $H_{\text{\rm red}} = H/H^{\times}$ the associated reduced monoid, and $\mathcal A (H)$ the set of atoms of $H$. A monoid $F$ is factorial with $F^{\times} = \{1\}$ if and only if it is free abelian. If this holds, then the set of primes $P \subset F$ is a basis of $F$, we write $F = \mathcal F (P)$, and every $a \in F$ has a representation of the form
\[
a = \prod_{p \in P} p^{\mathsf v_p (a)} \quad \text{with} \ \mathsf v_p (a) \in \mathbb N_0 \quad \text{and} \quad \mathsf v_p (a) = 0 \ \text{for almost all} \ p \in P \,.
\]
If $a \in F$, then $|a| = \sum_{p \in P} \mathsf v_p (a) \in \mathbb N_0$ is the length of $a$ and $\supp (a) = \{ p \in P \mid \mathsf v_p (a) > 0 \} \subset P$ is the support of $a$.

Let $G$ be an additively written abelian group and $G_0 \subset G$ a subset. Then $\langle G_0 \rangle \subset G$ denotes
the subgroup generated by $G_0$.    A family $(e_i)_{i \in I}$ of elements of $G$ is said to be \
{\it independent} \ if $e_i \ne 0$ for all $i \in I$ and, for every
family $(m_i)_{i \in I} \in \mathbb Z^{(I)}$,
\[
\sum_{i \in I} m_ie_i =0 \qquad \text{implies} \qquad m_i e_i =0 \quad \text{for all} \quad i \in I\,.
\]
The family $(e_i)_{i \in I}$ is called a {\it basis} for $G$ if
 $G = \bigoplus_{i \in I} \langle e_i \rangle$.

\smallskip
\noindent
{\bf Arithmetical concepts.} Our notation and terminology are consistent with \cite{Ge-HK06a}. We briefly gather some key notions. The free abelian monoid $\mathsf Z (H) = \mathcal F (\mathcal A (H_{\text{\rm red}}))$ is the factorization monoid of $H$, and the unique homomorphism $\pi \colon \mathsf Z (H) \to H_{\text{\rm red}}$ satisfying $\pi (u) = u$ for all $u \in \mathcal A (H_{\text{\rm red}})$ is the factorization homomorphism of $H$ (so $\pi$ maps a formal product of atoms onto its product in $H_{\text{\rm red}}$). For $a \in H$,
\[
\begin{aligned}
\mathsf Z_H (a) = \mathsf Z (a) & = \pi^{-1} (a) \subset \mathsf Z (H) \quad \text{is the set of factorizations of $a$, and} \\
\mathsf L_H (a) = \mathsf L (a) & = \{ |z| \mid z \in \mathsf Z (a) \} \quad \text{is the set of lengths of $a$} \,.
\end{aligned}
\]
Then $H$ is atomic (i.e., each non-unit can be written as a finite product of atoms) if and only if $\mathsf Z (a) \ne \emptyset$ for each $a \in H$, and $H$ is factorial if and only if $|\mathsf Z (a)|=1$ for each $a \in H$. Furthermore, for each $a \in H$,
$\mathsf L (a) = \{0\}$ if and only if $a \in H^{\times}$, and for all non-units the present definition coincides with the informal one given in the introduction. In particular, $\mathsf L (a) = \{1\}$ if and  only if $a \in \mathcal A (H)$. We denote by $\mathcal L (H) = \{\mathsf L (a) \mid a \in H\}$ the system of sets of lengths of $H$, and by
\[
\Delta (H) = \bigcup_{L \in \mathcal L (H) } \Delta (L) \qquad \text{the set of distances of $H$} \,.
\]
Distances occurring in sets of lengths $L$ with $2 \in L$ (in other words, in sets of lengths $\mathsf L (uv)$ with atoms $u,v \in H$) will play a central role. We define
\[
\daleth (H) = \sup \{ \min ( L \setminus \{2\}) \mid L \in \mathcal L (H) \ \text{with} \ 2 \in L \}
\]
and observe that $\daleth (H) \le 2+\sup \Delta (H)$.
Before we define catenary degrees we recall the concept of the distance between factorizations. Two factorizations  $z,\, z' \in \mathsf Z (H)$  can be written in
the form
\[
z = u_1 \cdot \ldots \cdot u_lv_1 \cdot \ldots \cdot v_m \quad
\text{and} \quad z' = u_1 \cdot \ldots \cdot u_lw_1 \cdot \ldots
\cdot w_n
\]  with
\[
\{v_1 ,\ldots, v_m \} \cap \{w_1, \ldots, w_n \} = \emptyset,
\]
where  $l,\,m,\, n\in \mathbb N_0$ and $u_1, \ldots, u_l,\,v_1, \ldots,v_m,\,
w_1, \ldots, w_n \in \mathcal A(H_\text{\rm red})$. Then $\gcd (z, z') = u_1 \cdot \ldots \cdot u_l$, and we call
$\mathsf d (z, z') = \max \{m,\, n\} = \max \{ |z \gcd (z,
z')^{-1}|, |z'\gcd (z, z')^{-1}| \}\in \mathbb N_0$ the {\it distance}
between $z$ and $z'$. It is easy to verify that  $\mathsf d \colon \mathsf Z (H) \times \mathsf Z (H) \to \mathbb N_0$ has all the usual properties of a metric.

Let $a \in H$ and
$N \in \mathbb N_0 \cup \{\infty\}$. A finite sequence $z_0, \ldots,
z_k \in \mathsf Z (a)$ is called an {\it  $N$-chain of
factorizations} if $\mathsf d (z_{i-1}, z_i) \le N$ for all $i \in
[1, k]$. We denote by  $\mathsf c (a)$   the smallest $N \in \mathbb N _0 \cup \{\infty\}$ such
      that any two factorizations $z,\, z' \in \mathsf Z (a)$ can be
      concatenated by an $N$-chain. Note that $\mathsf c (a) \le \sup \mathsf L (a)$, that $\mathsf c (a)=0$ if and only if $|\mathsf Z (a)|=1$, and if $|\mathsf Z (a)| > 1$, then $2 + \sup \Delta (\mathsf L (a)) \le \mathsf c (a)$.
Globalizing this concept we define
\[
\mathsf c(H) = \sup \{ \mathsf c(b) \, \mid b \in H\} \in \mathbb N_0 \cup
\{\infty\}
\]
as  the \ {\it catenary degree} of $H$. Then $\mathsf c (H)=0$ if and only if $H$ is factorial, and if $H$ is not factorial, then $2 + \sup \Delta (H)\le \mathsf c (H)$.

\smallskip
\noindent
{\bf Krull monoids.} A monoid homomorphism  $\varphi \colon H \to F$ is said to be a divisor homomorphism if $\varphi (a) \t \varphi (b)$ in $F$ implies that $a \t b$ in $H$ for all $a, b \in H$. A monoid $H$ is said to be a Krull monoid if one of the following equivalent properties is satisfied (see \cite[Theorem 2.4.8]{Ge-HK06a} or \cite{HK98}):
\begin{enumerate}
\item[(a)] $H$ is completely integrally closed and satisfies the ascending chain condition on divisorial ideals.

\item[(b)] $H$ has a divisor homomorphism into a free abelian monoid.

\item[(c)] $H$ has a divisor theory: this is a divisor homomorphism $\varphi \colon H \to F = \mathcal F (P)$ into a free abelian monoid such that for each $p \in P$ there is a finite set $E \subset H$ with $p = \gcd \big( \varphi (E) \big)$.
\end{enumerate}
Let $H$ be a Krull monoid. Then a divisor theory is unique up to unique isomorphism and the group  $\mathcal C (\varphi ) = \mathsf q (F)/\mathsf q (\varphi (H))$ depends only on $H$, and hence it is called the class group of $H$. We say that a class $g = [a]=a \mathsf q (\varphi (H)) \subset \mathsf q (F) \in \mathcal C (\varphi)$, with $a \in \mathsf q (F)$, contains a prime divisor if $g \cap P \ne \emptyset$. A domain $R$ is Krull if and only if its multiplicative monoid $R^{\bullet}=R\setminus \{0\}$ of non-zero elements is Krull. Thus Property (a) shows that every integrally closed noetherian domain is Krull. If $R$ is Krull with finite class group such that each class contains a prime divisor, then the same is true for regular congruence submonoids of $R$ (\cite[Section 2.11]{Ge-HK06a}). For monoids of modules which are Krull we refer the reader to \cite{Ba-Wi13a, Ba-Ge14b,Fa06a}.

Next we discuss a Krull monoid having a combinatorial flavor. It plays a universal role in all arithmetical studies of general Krull monoids. Let $G$ be an additive abelian group and $G_0 \subset G$ a subset. According to the tradition in combinatorial number theory, elements $S \in \mathcal F (G_0)$ will be called sequences over $G_0$ (for a recent presentation of their theory we refer to \cite{Gr13a}, and for an overview of their interplay with factorization theory we refer to \cite{Ge09a}). Let $S = g_1 \cdot \ldots \cdot g_l = \prod_{g \in G_0} g^{\mathsf v_g (S)} \in \mathcal F (G_0)$ be a sequence over $G_0$. Then $\mathsf h (S) = \max \{ \mathsf v_g (S) \mid g \in G_0 \}$ denotes the maximum multiplicity of $S$, $\sigma (S) = g_1 + \ldots + g_l \in G$ is the sum of $S$, and
\[
\Sigma (S) = \Big\{ \sum_{i \in I} g_i \mid \emptyset \ne I \subset [1,l] \Big\} = \Big\{ \sigma (T) \mid 1 \ne T \ \text{is a subsequence of $S$} \Big\} \subset G
\]
denotes the set of subsums of $S$. We say that $S$ is zero-sum free if $0 \notin \Sigma (S)$ and that $S$ is a zero-sum sequence if $\sigma (S)=0$. Obviously, $S$ is zero-sum free or a (minimal) zero-sum sequence if and only if $-S = (-g_1) \cdot \ldots \cdot (-g_l)$ has this property. The set
\[
\mathcal B (G_0) = \{ S \in \mathcal F (G_0) \mid \sigma (S)=0\} \subset \mathcal F (G_0)
\]
of all zero-sum sequences over $G_0$ is a submonoid of $\mathcal F (G_0)$, and since the embedding $\mathcal B (G_0) \hookrightarrow \mathcal F (G_0)$ obviously is a divisor homomorphism, Property (b) shows that $\mathcal B (G_0)$ is a Krull monoid. For each arithmetical invariant $*(H)$ defined for a monoid $H$, we write $* (G_0)$ instead of $*( \mathcal B (G_0) )$. This is the usual convention and will hardly lead to misunderstandings. In particular, we set $\mathcal L (G_0) = \mathcal L ( \mathcal B (G_0))$, $\daleth (G_0) = \daleth ( \mathcal B (G_0))$, and $\Delta (G_0)=\Delta (\mathcal B (G_0))$. The set $\mathcal A (G_0)=\mathcal A (\mathcal B (G_0))$ of atoms of $\mathcal B (G_0)$ is the set of minimal zero-sum sequences over $G_0$ and
\[
\mathsf D (G_0) = \sup \{ |A| \mid A \in \mathcal A (G_0) \} \in \mathbb N \cup \{\infty\}
\]
is the Davenport constant of $G_0$. If $G_0$ is finite, then $\mathcal A (G_0)$ is finite and hence $\mathsf D (G_0)< \infty$.

Suppose that $G$ is finite abelian, say $G \cong C_{n_1} \oplus \ldots \oplus C_{n_r}$ with $1 < n_1 \t \ldots \t n_r$. Then $1 + \sum_{i=1}^r (n_i-1) \le \mathsf D (G)$. We will use without further mention that equality holds for $p$-groups and for groups of rank $r \le 2$ (\cite[Chapter 5]{Ge-HK06a}). Furthermore, we will frequently use that,  if $\exp (G)+1 = n_r+1 = \mathsf D (G)$, then $r=2$ and $G \cong C_2 \oplus C_{n_2}$. If $S \in \mathcal F (G)$ is zero-sum free of length $|S|= \mathsf D (G)-1$, then $\Sigma (S)=G\setminus \{0\}$, and if $S \in \mathcal A (G)$ with $|S|=\mathsf D (G)$, then $\Sigma (S)=G$ (see \cite[Proposition 5.1.4]{Ge-HK06a}).

Suppose that $|G| \ge 3$. Then $\mathcal B (G)$ is a Krull monoid whose class group is isomorphic to $G$ and each class contains precisely one prime divisor (\cite[Proposition 2.5.6]{Ge-HK06a}). Furthermore, its arithmetic reflects the arithmetic of more general Krull monoids as it is summarized in the next proposition (for a proof see \cite[Section 3.4]{Ge-HK06a}).

\medskip
\begin{proposition} \label{2.1}
Let $H$ be a Krull monoid, $\varphi \colon H \to F = \mathcal F (P)$ a divisor theory, $G = \mathcal C(\varphi)$ the class group, and suppose that each class contains a prime divisor. Let $\widetilde{\boldsymbol \beta} \colon \mathcal F (P) \to \mathcal F (G)$ denote the unique homomorphism satisfying $\widetilde{\boldsymbol \beta} (p) = [p]$ for each $p \in P$, and let $\boldsymbol \beta = \widetilde{\boldsymbol \beta} \circ \varphi \colon H \to \mathcal B (G)$.
\begin{enumerate}
\item For each $a \in H$, we have $\mathsf L_H (a) = \mathsf L_{\mathcal B (G)} ( \boldsymbol \beta (a))$. In particular, $\mathcal L (H) = \mathcal L (G)$, $\Delta (H)=\Delta (G)$, and $\daleth (H) = \daleth (G)$.

\smallskip
\item We have $|G|\le 2$ if and only if $\mathsf D (G)\le 2$ if and only if $|L|=1$ for all $L \in \mathcal L (G)$.

\smallskip
\item If $|G| \ge 3$, then $\mathsf c (H) = \mathsf c (G)$.
    \end{enumerate}
\end{proposition}

\medskip
\section{On sets of lengths $L \in \mathcal L (G)$ having extremal properties} \label{3}
\medskip

In this section we mainly study sets of lengths of zero-sum sequences over finite abelian groups. We start by recalling two results from \cite{Ge-Gr-Sc11a}.

\medskip
\begin{lemma} \label{3.1}
Let $G$ be a finite abelian group with $|G| \ge 3$, say $G = C_{n_1} \oplus \ldots \oplus C_{n_r}$ with $1 < n_1 \t \ldots \t n_r$.
\begin{enumerate}
\item $\mathsf c(G) \leq \max \Big\{ \Big\lfloor\frac{1}{2} \mathsf D(G)+1 \Big\rfloor,\, \daleth (G) \Big\}$.

\item $\daleth (G) \ge  \max\{n_r,\,1+\sum_{i=1}^{r}\lfloor\frac{n_i}{2}\rfloor\}$.
\end{enumerate}
\end{lemma}

\begin{proof}
See Proposition 4.1 and Theorem 4.2 in \cite{Ge-Gr-Sc11a}.
\end{proof}

\medskip
\begin{lemma}\label{3.2}
Let $G$ be an abelian group with $|G| \ge 3$, and let $U,\,V\in \mathcal A(G)$ with $\max \mathsf L (UV) \ge 3$.
\begin{enumerate}
\item  Let $K \subset  G$ be a finite cyclic subgroup.  If $\sum_{h \in K} \mathsf v_h (UV) >  |K|$ and there exists a non-zero $g \in K$ such that $\mathsf v_{g}(U)>0$ and
       $\mathsf v_{-g}(V)>0$, then $\mathsf L (UV)\cap [3,|K|] \ne \emptyset$.

\smallskip
\item If $\mathsf L (UV)\cap [3, \ord (g)] = \emptyset$ for some  $g \in G$, then  $\mathsf v_g (U) + \mathsf v_{-g} (V) \le \ord (g)$.
\end{enumerate}
\end{lemma}

\begin{proof}
1. See Lemma 5.2 in \cite{Ge-Gr-Sc11a}.

\smallskip
2. Let $g \in G$.  If $\mathsf v_g (U)=0$ or $\mathsf v_{-g} (V)=0$, then the assertion is clear. If $\mathsf v_g (U)>0$ and $\mathsf v_{-g} (V)>0$, then the assertion follows from 1. with $K = \langle g \rangle$.
\end{proof}

\smallskip
Let $A$ be a zero-sum sequence over a finite abelian group $G$ and suppose that $0 \nmid A$. If
\[
A = U_1 \cdot \ldots \cdot U_k = V_1 \cdot \ldots \cdot V_l \,,
\]
where $k,l \in \mathbb N$ and $U_1, \ldots, U_k,V_1, \ldots, V_l \in \mathcal A (G)$, then obviously
\[
2 l \le \sum_{i=1}^l |V_i| = |A|= \sum_{i=1}^k |U_i| \le k \mathsf D (G) \,.
\]
These inequalities will be used implicitly in many of the forthcoming arguments.

\medskip
\begin{lemma} \label{3.3}
Let $G = C_4 \oplus C_4$. Then $\daleth (G) = \mathsf c (G) = 5$.
Moreover, if $U, V \in \mathcal A (G)$ with $\mathsf L (UV) \cap [2,
5] = \{2, 5\}$, then $\langle \supp (UV) \rangle = G$.
\end{lemma}

\begin{proof}
First, we prove the moreover statement. Let $U, V \in \mathcal A
(G)$ with $\mathsf L (UV) \cap [2, 5] = \{2, 5\}$. Then $\daleth
\bigl( \langle \supp (UV) \rangle \bigr) \ge 5$. Since for every
proper subgroup $K$ of $G$ we have $\daleth (K) \le \daleth (C_2
\oplus C_4) < \mathsf D (G)=5$ by Theorem {\bf A}, it follows
that $\langle \supp (UV) \rangle = G$.

Recall that $\mathsf D (G) = 7$, and note that it  suffices to prove
$\daleth (G) \le 5$, since then combing with \cite[Proposition 4.1.2
and Corollary 4.3]{Ge-Gr-Sc11a} yields  $5 \le \daleth (G) = \mathsf
c (G) \le  5$.

Let $U, V \in \mathcal A (G)$ with $\max \mathsf L (UV)
> 5$ be given. Since $\max \mathsf L (UV) \le \min \{|U|, |V| \} \le
\mathsf D (G) = 7$, it follows that $\min \{|U|, |V| \} \in
\{6,7\}$. We have to show that there exists a factorization $U V =
W_1 \cdot \ldots \cdot W_k$ with $W_1, \ldots, W_k \in \mathcal A
(G)$ and $k \in [3, 5]$. We start with two  special cases.

First, suppose that $V = -U$ and $|U| = 7$. Then \cite[Theorem
6.6.7]{Ge-HK06a} implies that $4 \in \mathsf L (UV)$, and thus the
assertion follows.

Second, suppose that there exist  $W_1, W_2 \in \mathcal A (G)$ such
that $W_1W_2 \t UV$, $5 \ge |W_1| \ge |W_2|$, and $|W_1 W_2| \ge 7$.
Then there exist $k \in \mathbb N_{\ge 3}$, $W_3, \ldots, W_k \in \mathcal
A (G)$ such that $UV = W_1 \cdot \ldots \cdot W_k$, and
\[
2(k-2) \le |W_3 \cdot \ldots \cdot W_k| = |UV| - |W_1W_2| \le 14 - 7
= 7
\]
implies $k \le 5$, and the assertion follows.

Assume to the contrary that $UV$ has no factorization of length $k
\in [3, 5]$.  Then none of the two special cases holds true. By
\cite[Lemma 3.6]{Ge-Gr09c}, $UV$ has a zero-sum subsequence $W_1 \in
\mathcal A (G)$ of length $|W_1| \in [2,4]$, and suppose that
$|W_1|$ is maximal. Then there is a factorization
\[
UV = W_1 \cdot \ldots \cdot W_k \quad \text{with} \quad k \ge 3
\quad \text{and} \quad  W_1, \ldots, W_k \in \mathcal A (G) \,.
\]
By assumption we have $k \ge 6$. Since $k=7$ would imply that $V =
-U$ and $|U| = 7$, it follows that $k=6$. We distinguish three
cases.

\medskip
\noindent CASE 1: \, $|W_1| = 2$.

Since $|W_1|$ is maximal, we get $|W_1| = \ldots = |W_k| = 2$, and
thus $|UV| \in \{12, 14\}$ and $V = -U$. Since we are not in the
first special case, it follows that $|U| = 6$, say $U = g_1 \cdot
\ldots \cdot g_6$.

If $\mathsf h (U) = 3$, say $g_1=g_2=g_3$, then $W = (-g_1)g_4g_5g_6
\in \mathcal A (G)$ with $W \t UV$, a contradiction. Thus $\mathsf h
(U) \le 2$. Since $\mathsf D (C_2^2) = 3$, $\supp (U)$ contains at
most two elements of order $2$, say $\ord (g_1) = \ldots = \ord
(g_4) = 4$. Since $\mathsf D (C_2 \oplus C_4) = 5$, it follows that
$\supp (U)$ generates $G$. Recall that the maximal size of a minimal
generating set of $G$ equals $\mathsf r^* (G) = 2$, and that every
generating set contains a basis (see \cite[Appendix A]{Ge-HK06a}).
Thus the elements of order $4$ contain a basis.

Suppose there are two elements with multiplicity two, say $g_1 =
g_3$ and $g_2 = g_4$. Then $(g_1, g_2)$ is a basis of $G$ and $g_5 =
a g_1 + b g_2$ with $a, b \in [1,3]$. Thus there is a zero-sum
subsequence $W \in \mathcal A (G)$ with $|W| > 2$, $W \t UV$ and
$\supp (W) \in \{g_5, g_1, g_2, - g_1, -g_2\}$, a contradiction.

Suppose there is precisely one element with multiplicity two, say
$g_1 = g_3$. Then $(g_1, g_2)$ is a basis of $G$ and there is an
element in $\supp (U)$ of the form $ag_1 + bg_2$ with $a \in [1,3]$
and $b \in \{1,3\}$. As above we obtain a zero-sum sequence $W$ with
$|W| > 2$, a contradiction.

Suppose that $\mathsf h (U) = 1$. Then $(g_1, g_2)$ is a basis of
$G$. Then there is one element in $\supp (U)$ which is not of the
form $\{ g_1+2g_2, 2g_1+g_2, 2g_1+2g_2\}$, and hence it has the form
$ag_1+bg_2$ with $a, b \in \{1,3\}$, and we obtain a contradiction
as above. \qed

\medskip
\noindent CASE 2: \, $|W_1| = 4$.

Since $k=6$, $|UV| \le 14$ and $|W_1| = 4$, we get $|W_2| = \ldots =
|W_6| = 2$ and  $|U| = |V| = 7$. By \cite[Example 5.8.8]{Ge-HK06a}, there exists a basis $(e_1, e_2)$
of $G$ such that
\[
U = e_1^3 \prod_{\nu=1}^4 (a_{\nu}e_1 + e_2)
\]
with $a_1, a_2, a_3 , a_4 \in [0,3]$ and $a_1+a_2+a_3+a_4 \equiv 1
\mod 4$. We set $X = \gcd (W_1, U)$. Then there are $X, Y, Z \in
\mathcal F (G)$ such that $U = X Z$, $V = Y (-Z)$ and $W_1 = X Y$.
After renumbering if necessary there are the following three cases:
\[
X = e_1^2, \ X = e_1 (a_1e_1+e_2) \quad \text{or} \quad X =
(a_1e_1+e_2)(a_2e_1+e_2) \,.
\]
If $X = e_1^2$, then $(-e_1)(a_{1}e_1 + e_2) \cdot \ldots \cdot
(a_{4}e_1 + e_2)$ is a minimal zero-sum subsequence of $UV$ of
length $5$, a contradiction.

Suppose that $X = (a_1e_1+e_2)(a_2e_1+e_2)$. If $a_3 \ne a_4$, then
$(a_3e_1+e_2)(-a_4e_1-e_2) e_1^3$ has a zero-sum subsequence of
length $4$ or $5$, a contradiction. Suppose that $a_3 = a_4$. Then
$(e_1, a_3e_1+e_2)$ is a basis, and after changing notation if
necessary we may suppose that $a_3= 0$. Then $(a_1e_1 +
e_2)(a_2e_1+e_2) \in \{ e_2(e_1+e_2), (2e_1+e_2)(-e_1+e_2)\}$. In
the first case $e_1^3 (e_1+e_2)(-e_2)$ and in the second case
$(-e_1)^3(-e_1+e_2)(-e_2)$ is a zero-sum subsequence of $UV$ of
length $5$, a contradiction.

Suppose that $X = e_1 (a_1e_1+e_2)$. If two of the $a_2, a_3, a_4$
are distinct, say $a_2 \ne a_3$, then  $(a_2e_1+e_2)(-a_3e_1-e_2)
e_1^2 (-e_1)^2$ has a zero-sum subsequence of length greater than
$2$, a contradiction. Suppose that $a_2 = a_3 = a_4$. Then $(e_1,
e_2'= a_2e_1+e_2)$ is a basis. Then $U = e_1^3 (e_1+e_2'){e_2'}^3$
and $e_1^3(e_1+e_2') (-e_2')$ is a minimal zero-sum subsequence of
$UV$ of length $5$, a contradiction.

\medskip
\noindent CASE 3: \,$|W_1| = 3$.

After renumbering if necessary we may suppose that $|U| = 7$, $|V|
\in \{6,7\}$, $|W_2| \in \{2,3\}$ and $|W_3| = \ldots = |W_6| = 2$.
By \cite[Example 5.8.8]{Ge-HK06a}, there exists a basis $(e_1, e_2)$ of $G$ such that
\[
U = e_1^3 \prod_{\nu=1}^4 (a_{\nu}e_1 + e_2)
\]
with $a_1, a_2, a_3 , a_4 \in [0,3]$ and $a_1+a_2+a_3+a_4 \equiv 1
\mod 4$. Thus $U$ has a subsequence $S$ of length $|S| = 4$ such
that $-S$ is a subsequence of $V$.

Suppose there are two distinct elements in $\{a_1e_1+e_2, \ldots,
a_4e_1+e_2\}$, say $(a_1e_1+e_2)$ and $(a_2e_1+e_2)$ such that
$(a_1e_1+e_2)(a_2e_1+e_2) \t S$. Then either
$(a_1e_1+e_2)(-a_2e_1-e_2)e_1^3$ or $(-a_1e_1-e_2)(a_2e_1+e_2)e_1^3$
contains a zero-sum subsequence of length $4$ or $5$, a
contradiction.

Now we suppose that this does not hold and distinguish three cases.

Suppose $S = e_1 (a_1e_1+e_2)(a_2e_1+e_2)(a_3e_1+e_2)$. Then
$a_1=a_2=a_3=a$, and since $(e_1, ae_1+e_2)$ is a basis of $G$, we
may suppose that $a= 0$. Then $U = e_1^3e_2^3(e_1+e_2)$ and
$(-e_1)e_2^3(e_1+e_2)$ is a zero-sum subsequence of $UV$ of length
$5$, a contradiction.

Suppose $S = e_1^2 (a_1e_1+e_2)(a_2e_1+e_2)$. Then $a_1=a_2$, and
since $(e_1, ae_1+e_2)$ is a basis of $G$, we may suppose that $a=
0$. Then $U = e_1^3 e_2^2 (ae_1+e_2) \bigl( (1-a)e_1+e_2 \bigr)$,
and either $(-e_2)e_1^3(ae_1+e_2)$ or $(-e_2)e_1^3\bigl(
(1-a)e_1+e_2 \bigr)$ has a zero-sum subsequence of length greater
than or equal to $4$, a contradiction.

Suppose $S = e_1^3 (a_1e_1+e_2)$. Again we may suppose that $a=0$
and find a zero-sum subsequence of $UV$ of length greater than or
equal to $4$, a contradiction.
\end{proof}

\medskip
\begin{lemma}\label{3.4}
Let $G$ be a finite abelian group with $|G| \ge 3$ and $U \in \mathcal A (G)$ with $|U|=\mathsf D (G)$.
\begin{enumerate}
\item If $g_1,g_2, h\in \supp(U)$, then $g_1\neq 2g_2$, and if $\ord(h)=2$, then $g_1\neq 2g_2+h$.

\smallskip
\item Let $g \in G$ with $\mathsf v_g(U)=k \ge1$ and $\ord(g)>2$ and suppose that $|\supp(U)|\ge2$. Then there exists some $W \in \mathcal A (G)$ such that $W\t (-g)^k g^{-k}U$ and  $|W| \in [3,  \mathsf D(G)-1]$.
\end{enumerate}
\end{lemma}

\begin{proof}
1. Assume to the contrary that  $g_1=2g_2$.  Since $|U| = \mathsf D (G)$, it follows that $\Sigma (g_1^{-1}U) = G \setminus \{0\}$. Hence there exists some $W \in \mathcal F (G)$ such that $W\t g_1^{-1}U$ and $\sigma (W)=-g_2$. If $g_2\t W$, then $g_1 g_2^{-1}W$ is a proper zero-sum subsequence of $U$, a contradiction.
If $g_2 \nmid W$, then $g_2W$ is a proper zero-sum subsequence of $U$, a contradiction.

Assume to the contrary that $\ord(h)=2$ and  $g_1=2g_2+h$. There exists some $W \in \mathcal F (G)$ such that  $W\t g_1^{-1}U$ and $\sigma(W)=-g_2+h$. If $g_2\t W$, then  $g_1g_2^{-1}W$ is a proper zero-sum subsequence of $U$, a contradiction.
If $g_2\nmid W$ and $h\t W$, then $g_2h^{-1}W$ is a proper zero-sum subsequence of $U$, a contradiction.
If $g_2\nmid W$ and $h\nmid W$, then $g_2hW$ is a proper zero-sum subsequence of $U$, a contradiction.

2. Since $|(-g)^k g^{-k}U|=|U|=\mathsf D(G)$,  there exists some $W \in \mathcal A (G)$ such that $W\t (-g)^kg^{-k}U$. It is easy to show that $|W|\neq 2$. Assume the contrary that  $|W|=\mathsf D(G)$. Then $W=(-g)^k g^{-k}U$ and $\ord(g)=2k>k+1$. Let $h\in \supp(U)\setminus \{g\}$. Since $\Sigma (h^{-1}U) = G \setminus \{0\}$,  there exists some $T \in \mathcal F (G)$ such that $T\t h^{-1}U$ and $\sigma (T)=(k+1)g$. If $g\nmid T$, then $g^{k-1}T$ is a  proper zero-sum subsequence of $U$, a contradiction. Suppose that $g\t T$, say $T=g^tT_0$ with $t \in [1,k]$ and  $T_0\t g^{-k}U$. Then    $W'=(-g)^{k-t+1}T_0$ is a zero-sum subsequence of $W=(-g)^k g^{-k}U$. Since $W$ is an atom, it follows that  $W'=W$. This implies that $T_0=g^{-k}U$, a contradiction to $T_0 \t T \t h^{-1}U$.
\end{proof}

\medskip
\begin{proposition} \label{3.5}
Let $G$ be a finite abelian group with $\mathsf D (G) \ge 5$. Then the following statements are equivalent{\rm \,:}
\begin{enumerate}
\item[(a)] $G$ is isomorphic  to $C_2 \oplus C_{2n}$ with $n \ge 2$.

\smallskip
\item[(b)] There exist  $U, V \in \mathcal A (G)$ with $\mathsf L (UV)= \{2, \mathsf D (G)-1, \mathsf D (G)\}$.
\end{enumerate}
\end{proposition}

\begin{proof}
(a)\, $\Rightarrow$\, (b) \ Let $(e_1, e_2)$ be a basis of $G$ with $\ord (e_1)=2$ and  $\ord(e_2)=2n$ with $n \ge 2$. Then $\mathsf D (G)=2n+1$, $U = e_1e_2^{2n-1}(e_1+e_2) \in \mathcal A (G)$, and $\mathsf L \big( U(-U) \big)=\{2,\mathsf D(G)-1,\mathsf D(G)\}$.

\smallskip
(b)\, $\Rightarrow$\, (a) \
Since $\mathsf D (G) \in \mathsf L (UV)$, it follows that $|U|=|V|=\mathsf D(G)$ and $V=-U$. Let $A \in \mathcal A (G)$ with  $A\t U(-U)$. Then $|A|\in \{2,3,\mathsf D(G)\}$, and  if $|A|=\mathsf D(G)$, then $U(-U) A^{-1}$ is an  atom of length $\mathsf D(G)$.
Since $|U|=\mathsf D (G)$, it follows that $\Sigma (U) =G$. By \cite[Theorem 6.6.3]{Ge-HK06a}, $G$ is neither  cyclic nor an elementary $2$-group.  Therefore,  $|\supp (U)|\ge 2$ and may write
$U=g_1^{k_1} \cdot \ldots \cdot g_s^{k_s}$ with $g_1, \ldots, g_s \in G$ pairwise distinct, $s \ge 2$, $k_1, \ldots, k_s \in \mathbb N$, and $\ord(g_1)>2$.

We continue with the following assertion.

\begin{enumerate}
\item[{\bf A1.}\,] If $\ord{g_i}>2$ for some $i \in [1,s]$, then there exist distinct elements $h_1, h_2 \in \supp (U)$ such that  $g_i=h_1+h_2$.

\item[{\bf A2.}\,] If $\ord{g_i}>2$ for some $i \in [1,s]$, then there exist distinct elements $f_1, f_2 \in \supp (U)$ such that  $g_i=f_1+f_2$, $\ord (f_1)>2$, and $\ord (f_2)=2$.
\end{enumerate}

\smallskip
{\it Proof of \,{\bf A1}}.\, Let $i \in [1,s]$ with $\ord (g_i)>2$. By Lemma \ref{3.4}.2, there exists an atom $W \in \mathcal A (G)$ such that  $W \t g_i^{-k_i}(-g_i)^{k_i}U$ and $|W| \in [3, \mathsf D(G)-1]$. Then $W \t U(-U)$ infers $|W|=3$. Thus $W=(-g_i)^2h_1$ with $h_1\t U$ or  $W=(-g_i)h_1h_2$ with $h_1h_2\t U$. In the first case $h_1=2g_i$, a contradiction to Lemma \ref{3.4}.1. Thus the second case holds, and again  Lemma \ref{3.4}.1 implies that   $h_1$ and $h_2$ are distinct. \qed

\smallskip
{\it Proof of \,{\bf A2}}.\, Let $i \in [1,s]$ with $\ord (g_i)>2$, say $i=1$. By {\bf A1}, we may assume that $g_1=g_2+g_3$ and hence
  $\ord(g_2)>2$ or $\ord(g_3)>2$.
Assume to the contrary that $\ord(g_2)>2$ and $\ord(g_3)>2$. Again by {\bf A1}, there are distinct $h_1, h_2 \in \supp (U)$ such that $g_2=h_1+h_2$. Then $\{h_1,h_2\}\cap \{g_1,g_3\}\neq \emptyset$ otherwise $(-g_1)h_1h_2g_3$ would be an atom of length $4$ dividing $U(-U)$. Since $g_2=g_1+h_i$ with $i \in [1,2]$ cannot hold, we obtain that $g_2=g_3+h$ with $h \in \{h_1, h_2\}$. Repeating the argument we infer that  $g_3=g_2+h'$ with $h' \in \supp (U)$. It follows that $h+h'=0$ which implies that $h=h'$, $\ord(h)=2$, and $g_1=2g_3+h$,  a contradiction to Lemma \ref{3.4}.1. \qed

\smallskip
Since $\ord(g_1)>2$,  by {\bf A2} we may assume that   $g_1=g_2+g_3$ with $\ord(g_2)>2$ and  $\ord(g_3)=2$.
If $s=3$, then $G = \langle \supp (U) \rangle = \langle g_1, g_2, g_3 \rangle = \langle g_2, g_3 \rangle \cong C_2\oplus C_{2n}$ with  $n \ge 2$.
Assume to the contrary  that  $s\ge4$. We distinguish two cases.

\smallskip
\noindent
CASE 1: \, There is an $i \in [4,s]$ such that $\ord (g_i) > 2$, say $i=4$.

By {\bf A2}, we may assume  $g_4=h_1+h_2$ with $\ord(h_1)>2$ and $\ord(h_2)=2$. We assert that $\{g_1,g_2,g_3\}\cap \{g_4, h_1,h_2\}=\emptyset$. Assume to the contrary that this does not hold. Taking into account the order of the elements and that $|\{g_1, \ldots, g_4\}|=4$ we have to consider the following three cases:
\begin{itemize}
\item If $h_2=g_3$, then $g_1+g_4=g_2+h_1$ and hence $(-g_1)(-g_4)g_2h_1$ is an atom of length $4$ dividing $U(-U)$, a contradiction.

\item If $h_1=g_1$, then $(-g_4)g_2g_3h_2$ is an atom of length $4$ dividing $U(-U)$, a contradiction.

\item If $h_1=g_2$, then $(-g_4)g_1g_3h_2$ is an atom of length $4$  dividing $U(-U)$, a contradiction.
\end{itemize}
Thus we have $\{g_1,g_2,g_3\}\cap \{g_4, h_1,h_2\}=\emptyset$. Obviously,  $(-g_1)g_2g_3$, $g_1(-g_2)(-g_3)$, $(-g_4)h_1h_2$, and $g_4(-h_1)(-h_2)$ are atoms of length $3$ dividing $U(-U)$, and therefore their product
$(-g_1)g_2g_3\cdot g_1(-g_2)(-g_3)\cdot (-g_4)h_1h_2\cdot g_4(-h_1)(-h_2)$ divides  $U(-U)$, a contradiction to $\mathsf L \big( U(-U) \big)= \{2, \mathsf D (G)-1, \mathsf D (G)\}$.

\smallskip
\noindent
CASE 2: \,  $\ord(g_i)=2$ for all $i\in [4,s]$.

If $k_1\ge 2$ and $k_2\ge 2$, then $(-g_1)^2g_2^2$ is an atom of length $4$ dividing $U(-U)$, a contradiction. Thus $k_1=1$ or $k_2=1$, say $k_1=1$. Since $g_1+k_2g_2+g_3+\ldots+g_s=\sigma (U)=0$ and $(k_2+1)g_2=g_4+\ldots+g_s$,
it follows that $\ord (g_2)=2(k_2+1)$. Thus
\[
\exp (G)+s-3 \le \mathsf D (G)=|U|=k_2+s-1 \le 2(k_2+1)+s-4\le \exp (G)+s-4 \,,
\]
a contradiction.
\end{proof}

\medskip
\begin{proposition} \label{3.6}
Let $G$ be a finite abelian group with $\mathsf D (G) \ge 5$. Then there are no  $U, V \in \mathcal A (G)$ with $\mathsf L (UV)= \{2, \mathsf D (G)-1 \}$ and $|U|=|V|=\mathsf D (G)$.
\end{proposition}

\begin{proof}
Assume to the contrary that $U, V \in \mathcal A (G)$ with these properties do exist.
If   $A\in \mathcal A(G)$ with $A\t UV$, then  $|A|\in\{2,3,4,\mathsf D(G)\}$.
Since $\Sigma (U) = \Sigma (V)=G$ and $G$ cannot be cyclic, it follows that $|\supp (U)|\ge 2$ and $|\supp (V)|\ge 2$.
We distinguish two cases.

\smallskip
\noindent
CASE 1:\,  For all  $A\in \mathcal A(G)$ with $A \t UV$ we have that $|A|\in\{2,3,\mathsf D(G)\}$.

Then $UV$ has a factorization of the form  $UV=W_1 \cdot \ldots \cdot W_{\mathsf D(G)-1}$ where $W_1, \ldots, W_{\mathsf D(G)-1} \in \mathcal A (G)$,  $|W_{\mathsf D(G)-2}|=|W_{\mathsf D(G)-1}|=3$, and all the other $W_i$ have length $2$. Thus we may set
\[
U=g_1 \cdot\ldots \cdot g_l h_1h_4h_5 \quad \text{and} \quad V=(-g_1)\cdot \ldots \cdot  (-g_l)(-h_2)(-h_3)(-h_6) \,,
\]
where $l \in \mathbb N$, $g_1, \ldots, g_l, h_1, \ldots, h_6 \in G$ not necessarily distinct such that $h_1=h_2+h_3$, $h_6=h_4+h_5$ and $\{h_1,h_4,h_5\}\cap \{h_2,h_3,h_6\}=\emptyset$.

Since $V(-g_6)^{-1}$ is a zero-sum free sequence of length $\mathsf D(G)-1$, there exists a subsequence $T$ of  $V(-g_6)^{-1}$ such that $\sigma (T)=h_4$. Thus $Th_5(-h_6)$ is a zero-sum subsequence of $UV$.
Since $T(-h_6)$ is zero-sum free, we obtain that $Th_5(-h_6)$ is an atom of length $|Th_5(-h_6)|\in \{3,\,\mathsf D(G)\}$. If $|Th_5(-h_6)|=3$, then $|T|=1$ which implies that $h_4\t V$. If $|Th_5(-h_6)|=\mathsf D(G)$, then $Th_5(-h_6)=V$ which implies that $h_5\t V$. Therefore, we obtain that $h_4\in \supp(U)\cap \supp(V)$ or $h_5\in \supp(U)\cap \supp(V)$.
Without loss of generality, we  assume that $h_4\in \supp(U)\cap \supp(V)$.

We start with the following assertions.

\begin{enumerate}
\item[{\bf A1.}\,] $\ord (h_4) > 2$, $(-2h_4)^2\t U$, and $(-2h_4)^2\t V$.

\item[{\bf A2.}\,] If $g_i\not\in \{h_1,\ldots,h_6\}$  for some $i\in [1,l]$, then $\ord(g_i)=2$.

\end{enumerate}

\smallskip
{\it Proof of \,{\bf A1}}.\, Since  $-g_i\neq h_4$ for all $i\in [1,l]$, we obtain that $h_4\in \{-h_2,-h_3,-h_6\}$ which implies that  $\ord(h_4)> 2$.

By the symmetry of $U$ and $V$, we only need to prove that $(-2h_4)^2\t U$.
For any $g\in \supp(U)$ with $g\neq h_4$, consider the sequence $Ug^{-1}h_4$. Since $|Ug^{-1}h_4|=\mathsf D(G)$ and $Ug^{-1}h_4\t UV$, there exists a atom $A$ with $A\t Ug^{-1}h_4 $ and $|A|=3 $ by Lemma \ref{3.4}.2. Note that $h_4^2\t A$. Therefore, $-2h_4\t Ug^{-1}$.
If $\ord(h_4)=3$, then $h_4=-2h_4$ and $A=h_4^3$ which implies that $h_4^2\t U$.
If $\ord(h_4)>3$, then $-2h_4\neq h_4$. Thus we can choose $g=-2h_4$ which implies that $(-2h_4)^2\t U$.  \qed

\smallskip
{\it Proof of \,{\bf A2}}.\, Let $i \in [1,l]$ such that $g_i\not\in \{h_1,\ldots,h_6\}$ and  assume to the contrary that $\ord(g_i)>2$. Let $v=\mathsf v_{g_i}(U)=\mathsf v_{-g_i}(V)$. By Lemma \ref{3.4}.2,  we obtain that there exists an atom $W$ such that $W\t g_i^{-v}(-g_i)^vU$ and $|W|=3$. By Lemma \ref{3.4}.1, $\mathsf v_{-g_i}(W)=1$. Thus there exist $f_1, f_2\in \supp(U)$ such that $g_i=f_1+f_2$. Obviously, $f_1f_2\nmid g_1\cdot \ldots \cdot  g_l$. If $f_1\t g_1\cdot \ldots \cdot g_l$ and $f_2\t h_1h_4h_5$, then $f_2=g_i+(-f_1)$. Since $f_2$ is an element of an atom of length 3, we can substitute $f_2$ by $g_i(-f_1)$ to get an atom of length $4$, a contradiction. Therefore, $f_1f_2\t h_1h_4h_5$. If $f_1f_2=h_4h_5$, then $g_i=h_6$, a contradiction. By the symmetry of $h_4$ and $h_5$, we only need to consider $f_1f_2=h_1h_4$. Since $(-g_i)h_1h_4$ is an atom, $g_ih_5(-h_2)(-h_3)(-h_6)$ can only be a product of two atoms which implies that $g_i\in \{h_2,h_3,h_6\}$, a contradiction.    \qed

\smallskip
We continue with the following two subcases.

\smallskip
 \noindent
CASE 1.1:\, $\ord(h_4)>3$.

By {\bf A1}, we have $(-2h_4)^2\t U$ and $(-2h_4)^2\t V$.
Thus, for any $i \in [1,l]$, $g_i \ne -2h_4$ and $-g_i \ne -2h_4$. Therefore, we obtain that $h_1=h_5= -2h_4$ and $-h_6=h_4$, and hence $-h_2=-h_3=-2h_4$.
 Then $-2h_4=h_1=h_2+h_3=4h_4$ which implies that $\ord(h_4)=6$.
By {\bf A2}, we obtain that $U=g_1\cdot \ldots \cdot g_l h_4 (-2h_4)^2$ with $\ord(g_1)=\ldots =\ord(g_l)=2$ and $\ord(h_4)=6$. Since $\mathsf D(G)\ge5 $, we get $l\ge 2$. Then $6+l-1=\exp(G)+l-1\le \mathsf D(G)=|U|=l+3$, a contradiction

\smallskip
 \noindent
CASE 1.2:\,  $\ord(h_4)=3$.

By {\bf A1}, $h_4^2\t U$ and $h_4^2\t V$. Thus $h_4^4\t h_1h_4h_5(-h_2)(-h_3)(-h_6)$. But $h_4^6\neq  h_1h_4h_5(-h_2)(-h_3)(-h_6)$. Without loss of generality, we can assume that $h_4=h_5=-h_3=-h_6$, $h_1\neq h_4$, and $-h_2\neq h_4$.

Suppose $\mathsf v_{h_1}(U)>1$ and $\mathsf v_{-h_2}(V)>1$. Then $-h_1\t V$ and $h_2\t U$. Since $h_1+(-h_2)+h_4=0$, then $(-h_1)h_2h_4^2$ is an atom of length $4$ and divides $UV$, a contradiction.
Thus, by symmetry, we may assume that $\mathsf v_{-h_2}(U)=1$. Then $h_2\nmid U$.
Gathering the $g_i$'s, which are equal to $h_1$ and renumbering if necessary we obtain that
$U=g_1\cdot \ldots \cdot g_s h_1^v h_4^2$ with $s \in [0,l]$, $v=\mathsf v_{h_1}(U)$, $\ord(g_1)=\ldots=\ord(g_s)=2$ (by {\bf A2}), and $\ord(h_4)=3$.

If $s=0$, then $vh_1=h_4$ and $G=\langle h_1\rangle$, a contradiction.

Suppose $s\ge 1$.
Since $2\sigma (U)=2vh_1+4h_4=0$, we obtain that $2h_4=2vh_1$ and $6vh_1=0$.
Since $(-h_1)^{v-1}h_4^2 \t V$, we obtain that $(-h_1)^{v-1}h_4^2$ is zero-sum free and, for any $j \in [v+1, 2v]$, we have $j h_1 \in \Sigma \big( (-h_1)^{v-1}h_4^2 \big)$. This implies that $\ord(h_1)>2v$.
Thus $6v/\ord(h_1)<3$ which implies that $\ord(h_1)=3v$ or $6v$. But $\sigma (h_1^vh_4^2)=3vh_1\neq 0$, hence $\ord(h_1)=6v=\exp(G)$.
Therefore,
\[
\exp(G)+s-1\le \mathsf D(G)=|U|=s+v+2\le 6v+s-3\le \exp(G)+s-3 \,,
\]
a contradiction.

\bigskip
\noindent
CASE 2:\, There exists an $A\in \mathcal A(G)$ with $A\t UV$ and $|A|=4$.

Then $UV$ has a factorization of the form
$UV=W_1 \cdot \ldots \cdot  W_{\mathsf D(G)-1}$ where $W_1, \ldots, W_{\mathsf D(G)-1} \in \mathcal A (G)$,  $|W_{\mathsf D(G)-1}|=4$, and all the other $W_i$ have length $2$.
Conversely, note  that for any $A'\in \mathcal A(G)$ with $A'\t UV$ and $|A'|=4$, $UV(A')^{-1}$ can only be a product of atoms of length $2$ because $\mathsf D(G)\ge 5$.
We start with the following assertion.

\begin{enumerate}
\item[{\bf A3.}\,]  If $g\in \supp(U)\cap \supp(V)$, then $\ord(g)=2$.
\end{enumerate}

\smallskip
{\it Proof of \,{\bf A3}}.\, Let $U=g_1\cdot \ldots \cdot g_lh_1h_2$ and $V=(-g_1)\cdot \ldots \cdot (-g_l) h_3h_4$ where $l \in \mathbb N$, $g_1, \ldots, g_l, h_1, \ldots, h_4 \in G$ not necessarily distinct and $W_{\mathsf D(G)-1}=h_1h_2h_3h_4$.
Let $g\in \supp(U)\cap \supp(V)$. If $g = g_i$ for some $i \in [1,l]$, then $g \t V$, $-g \t V$, and hence $g=-g$. If $g = -g_i$ for some $i \in [1,l]$, then $g \t U$, $-g \t U$, and hence $g=-g$.

It remains to consider the case where $g \in \{h_1, h_2\} \cap \{h_3, h_4\}$, say
$h_1=h_3=g$. We  will show that this is not possible.
 Assume this holds, choose an $h \in \supp (U) \setminus \{g\}$, and consider the sequence $X = h^{-1}U h_3$. Since $|X|=\mathsf D(G)$,  there exists an atom $A \in \mathcal A (G)$ such that $A\t X$  and $|A| \in [3,4]$. Note that   $h_1h_3\t A$.

Suppose that $|A|=3$. Then there exists $h'\in \supp ( U)$ such that $h'=-h_1-h_3=-2g$. If $h'=h_2$, then $h_1h_2h_3$ ia a proper zero-sum sequence of $W_{\mathsf D(G)-1}$, a contradiction. Otherwise, $h' \in \{g_1, \ldots, g_l\}$, hence  $-h'=2g \in \supp( V)$, a contradiction to $h_3=g \in \supp ( V)$ (recall Lemma \ref{3.4}).

Suppose that $|A|=4$. If $h_2 \t A$, then (note that $h_1+h_2+h_3+h_4=0$) $A=h_1h_2h_3h_4$ and $-h_4 \t V$, a contradiction.  Thus $h_2 \nmid A$, and the similar argument shows that
$h_4\nmid A$. This implies that $A=g_ig_jh_1h_3$ with $i,j \in [1,l]$. Then $A$ and $A'=(-g_i)(-g_j)h_2h_4$ are two atoms of length $4$ dividing $UV$, a contradiction.
\qed

\medskip
Clearly, there are precisely two possibilities for $U$ and $V$ which will be discussed in the following two subcases.

 \smallskip
 \noindent
CASE 2.1:\, $
U=g_1^{k_1}\cdot \ldots \cdot g_l^{k_l}h_1^{r_1-1}h_2^{r_2-1}h_3^{r_3-1}h_4^{r_4-1}h_1h_2 \quad \text{ and}$
\[
  V=(-g_1)^{k_1}\cdot \ldots \cdot (-g_l)^{k_l}(-h_1)^{r_1-1}(-h_2)^{r_2-1}(-h_3)^{r_3-1}(-h_4)^{r_4-1}(-h_3)(-h_4) \,, \ \text{where}
\]
$k_1, \ldots, k_l,r_1,\ldots, r_4 \in \mathbb N$, $g_1, \ldots , g_l, h_1,\ldots,h_4$ are pairwise distinct, and $h_1h_2(-h_3)(-h_4) = W_{\mathsf D(G)-1}$.

 We start with the following assertions.

\begin{enumerate}
\item[{\bf A4.}\,] For each $g_i \in \{g_1,\ldots, g_l\}$ with $\ord (g_i)>2$, we have that $g_i=h_1+h_2$.

\item[{\bf A5.}\,] $r_1=r_2=r_3=r_4=1$.

\item[{\bf A6.}\,] $\ord(g_i)=2$ for all $i \in [1,l]$.
\end{enumerate}

\smallskip
{\it Proof of \,{\bf A4}}.\,  Let $i \in [1,l]$ such that  $\ord(g_i)>2$, say $i=1$.
Then, by Lemma \ref{3.4}.2,  there  exists an atom $A \in \mathcal A (G)$ such that   $A\t (-g_1)^{k_1}g_1^{-k_1}U$ and $|A| \in [3,4]$. We distinguish four subcases depending on the multiplicity of $\mathsf v_{-g_1} (A)$ and on $|A|$.

Suppose that $\mathsf v_{-g_1}(A)=3$.  Then $|A|=4$. Since $A\nmid UV(W_{\mathsf D(G)-1})^{-1}$, we must have that $h_1\t A$ or $h_2\t A$. Then $A=(-g_1)^3h_i$ with $i\in [1,2]$, say $i=1$. It follows that  $A'=g_1^3h_2(-h_3)(-h_4)$ is a zero-sum subsequence of $UV(A)^{-1}$ which implies that $A'$ is a product of atoms of length $2$, a contradiction.

Suppose that $\mathsf v_{-g_1}(A)=2$.  Then $|A|=4$ by Lemma \ref{3.4}.1. Since $A\nmid UV(W_{\mathsf D(G)-1})^{-1}$, we must have that $h_1\t A$ or $h_2\t A$.
If $h_1h_2\t A$, then $A=(-g_1)^2h_1h_2$, and hence  $A'=g_1^2(-h_3)(-h_4)$ is an  atom of length $4$ and divides $UVA^{-1}$, a contradiction.
By symmetry,  we may assume that $h_1\t A$ and $h_2\nmid A$. Thus  $A = (-g_1)^2h_1h$, where  $h \in \{g_2,\ldots,g_l,h_1,h_3,h_4\}$ and $(-h)(-h_3)(-h_4)\t V$. Therefore $A'=g_1^2 h_2(-h_3)(-h_4)(-h)$ is  a zero-sum subsequence of $UVA^{-1}$ which implies that $A'$ is a product of three atoms of length $2$,  a contradiction.

Suppose that $\mathsf v_{-g_1}(A)=1$ and  $|A|=4$. Since $A\nmid UV(W_{\mathsf D(G)-1})^{-1}$, we must have that $h_1\t A$ or $h_2\t A$.  If $h_1h_2\t A$, then $A=(-g_1)h_1h_2h$, where   $h \in \{g_2,\ldots,g_l,h_1,h_2,h_3,h_4\}$ and $(-h)(-h_3)(-h_4)\t V$. Hence  $A'=g_1(-h_3)(-h_4)(-h)$ is an  atom of length $4$ and divides  $UVA^{-1}$, a contradiction.
By symmetry,  we may assume that $h_1\t A$ and $h_2\nmid A$. Thus $A=(-g_1)h_1hh'$ where  $h, h' \in \{g_2,\ldots,g_l,h_1,h_3,h_4\}$ and $(-h)(-h')(-h_3)(-h_4)\t V$.  Thus $A'=g_1(-h)(-h')h_2(-h_3)(-h_4)$ is a  zero-sum subsequence of $UVA^{-1}$ which implies that  $A'$ is a product of three atoms of length $2$,  a contradiction.

Suppose that $\mathsf v_{-g_1}(A)=1$ and  $|A|=3$. Since $A\nmid UV(W_{\mathsf D(G)-1})^{-1}$, we must have that $h_1\t A$ or $h_2\t A$. If $h_1h_2\nmid A$, by symmetry we may assume that $h_1\t A$ and $h_2\nmid A$. Thus $A=(-g_1)h_1h$, where $h\in \{g_2,\ldots,g_l,h_1,h_3,h_4\}$ and $(-h)(-h_3)(-h_4)\t V$. It follows that  $A'= g_1(-h)h_2(-h_3)(-h_4)$ is a zero-sum subsequence of $UVA^{-1}$ which implies that $A'$ is a product of two atoms, a contradiction.
Therefore, $h_1h_2\t A$ and $g_1=h_1+h_2$. \qed

\smallskip
{\it Proof of \,{\bf A5}}.\, By symmetry it is sufficient to show that $r_3=1$. Assume to the contrary  that $r_3\ge2$. We proceed in several steps.

{\bf (i)} In the first step we will show that $h_3=g_i+h_1$ for some $i \in [1,l]$.

By Lemma \ref{3.4}.2, there  exists an atom $A \in \mathcal A (G)$ such that  $A\t h_3^{-(r_3-1)}(-h_3)^{r_3-1}U$ and $|A|\in[3,4]$. We distinguish four subcases depending on the multiplicity of $\mathsf v_{-g_1} (A)$ and on $|A|$.

Suppose that $\mathsf v_{-h_3}(A)=3$.  Then $|A|=4$. Since $A\nmid UV(W_{\mathsf D(G)-1})^{-1}$, we must have that $h_1\t A$ or $h_2\t A$. Then $A=(-h_3)^3h_i$ with $i\in [1,2]$, say $i=1$. Therefore    $A'=h_3^2h_2(-h_4)$ is an atom of length $4$ and divides $UVA^{-1}$, a contradiction.

Suppose that  $\mathsf v_{-h_3}(A)=2$.  Then $|A|=4$ by Lemma \ref{3.4}.1. Since $A\nmid UV(W_{\mathsf D(G)-1})^{-1}$, we must have that $h_1\t A$ or $h_2\t A$.
If $h_1h_2\t A$, then $A=(-h_3)^2h_1h_2$ which implies that $h_3=h_4$, a contradiction.
By symmetry,  we may assume that $h_1\t A$ and $h_2\nmid A$. Thus  $A = (-h_3)^2h_1h$, where  $h \in \{g_1,\ldots,g_l,h_1,h_4\}$ and $(-h)(-h_3)(-h_4)\t V$. Therefore $A'=h_3^2 h_2(-h_3)(-h_4)(-h)$ is  a zero-sum subsequence of $UVA^{-1}$ which implies that $A'$ is a product of three atoms of length $2$,  a contradiction.

Suppose that $\mathsf v_{-h_3}(A)=1$ and  $|A|=4$. Since $A\nmid UV(W_{\mathsf D(G)-1})^{-1}$, we must have that $h_1\t A$ or $h_2\t A$.  If $h_1h_2\t A$, then $A=(-h_3)h_1h_2h$ with $h\in \{g_1,\ldots,g_l,h_1,h_2,h_4\}$ and $-h\t V$.   Hence $h=-h_4$ and $-h=h_4\t V$, a contradiction.
By symmetry,  we may assume that $h_1\t A$ and $h_2\nmid A$. Thus $A=(-h_3)h_1hh'$ where  $h, h' \in \{g_1,\ldots,g_l,h_1,h_4\}$ and $(-h)(-h')(-h_3)(-h_4)\t V$.  Thus $A'=(-h)(-h')h_2(-h_4)$ is an atom of length $4$ and divides $UVA^{-1}$,  a contradiction

Suppose that $\mathsf v_{-h_3}(A)=1$ and  $|A|=3$. Since $A\nmid UV(W_{\mathsf D(G)-1})^{-1}$, we must have that $h_1\t A$ or $h_2\t A$.
If $h_1h_2\t A$, then $A=(-h_3)h_1h_2$ which implies that $h_4=0$, a contradiction. Thus, by symmetry, we may assume that $h_1\t A$ and $h_2\nmid A$. Then $A=(-h_3)h_1h$, where $h\in \{g_1,\ldots,g_l,h_1,h_4\}$.
By Lemma \ref{3.4}.1, we obtain that $h\not\in \{h_1,h_4\}$.
Therefore, $h_3=h_1+g_i$ for some $i\in [1,l]$.

\smallskip
{\bf (ii)} In the second step we will show that $\ord (g_j)=2$ for all $j \in [1,l]$ and $r_1=1$.

Assume to the contrary that there is a $j \in [1,l]$ such that
$\ord(g_j)>2$. Then $g_j=h_1+h_2$ by {\bf A4} and hence  $A_1=g_j(-h_3)(-h_4)$ and $A_2=g_i(-h_3)h_1$ are two atoms of length $3$ and divide $UV$. It follows that $UV(A_1A_2)^{-1}$ is a product of atoms of length $2$, but $\mathsf v_{h_2}(UV(A_1A_2)^{-1})=\mathsf v_{h_2}(UV)>\mathsf v_{-h_2}(UV)=\mathsf v_{-h_2}(UV(A_1A_2)^{-1}) $, a contradiction.

Assume to the contrary that  $r_1\ge 2$. Since $\ord(g_i)=2$ and $h_3=h_1+g_i$, we obtain that  $h_1^2(-h_3)^2$ and $(-h_1)h_3g_i$ are two atoms and divide $UV$, a contradiction.

\smallskip
{\bf (iii)} In the third step we show that $\ord (h_3)=4r_3$.

Consider the sequence $X=U(h_3^{r_3-1}h_1)^{-1}(-h_3)^{r_3}$.
Since $|X|=\mathsf D(G)$, there exists an atom $A\in\mathcal A(G)$ with $A\t X$ and  $|A|\in \{3,4,\mathsf D(G)\}$. We distinguish three subcases depending on $|A|$.

Suppose that $|A|=3$.  Since $A\nmid UV(W_{\mathsf D(G)-1})^{-1}$ and $r_1=1$,  we must have that $h_2\t A$. Thus  $A=(-h_3)h_2h$, where $h\in  \{g_1,\ldots,g_l,h_2,-h_3\}$. By Lemma \ref{3.4}.1, $h\not\in \{h_2,-h_3\}$. Therefore, $A$ and $A'=(-h_3)h_1g_i$ are two atoms of length $3$ and divide $UV$. It follows that $UV(AA')^{-1}$ is a product of  atoms of length $2$ but $\mathsf v_{h_4}(UV(AA')^{-1})<\mathsf v_{-h_4}(UV(AA')^{-1})$, a contradiction.

Suppose that $|A|=4$. Then $UVA^{-1}$ is a product of atoms of length $2$, but $r_1=1=\mathsf v_{h_1}(UVA^{-1})>\mathsf v_{-h_1}(UVA^{-1})=0$, a contradiction.

Suppose that  $|A|=\mathsf D(G)$. Then $A=X$ and hence $h_1=-(2r_3-1)h_3$. By steps {\bf (i)} and {\bf (ii)}, we obtain that $g_i=2r_3h_3$ and  $4r_3h_3=0$. Since $Uh_1^{-1}$ is zero-sum free and for each $j\in [1,2r_3-1]$, $jh_3\in \Sigma (Uh_1^{-1})$, then $\ord(h_3)>2r_3$ which implies that  $\ord(h_3)=4r_3$.

\smallskip
{\bf (iv)} In the final  step we will obtain a contradiction to our assumption that $r_3 \ge 2$.
Clearly, similar arguments as given in the  steps {\bf (i)},{\bf (ii)}, and {\bf (iii)} show that  $r_2\ge 2$ implies that  $\ord(h_2)=4r_2$,  and that  $r_4\ge 2$ implies that $\ord(h_4)=4r_4$.
We proceed with the following four subcases depending on $r_2$ and $r_4$.

Suppose that  $r_2\ge 2$ and  $r_4\ge 2$. Since $h_3=h_1+g_i$ and $\ord(g_i)=2$, we obtain that $h_2=h_4+g_i$ and hence $2h_2=2h_4$. Therefore, $(-h_2)h_4g_i$ and $h_2^2(-h_4)^2$ are two atoms and divide $UV$, a contradiction.

Suppose that  $r_2=r_4=1$.  Then $U=g_1\cdot \ldots \cdot g_lh_3^{r_3-1}h_1h_2$ and
\[
\exp(G)+l-1\le\mathsf D(G)=l+r_3+1< l-1+4r_3\le \exp(G)+l-1 \,,
\]
a contradiction.

Suppose that $r_2=1$ and $r_4\ge2$. Then $U=g_1\cdot \ldots \cdot g_lh_3^{r_3-1}h_4^{r_4-1}h_1h_2$ and
\[
\exp(G)+l-1\le\mathsf D(G)=l+r_3+r_4<l-1+ 4\max\{r_3,r_4\}\le\exp(G)+l-1 \,,
\]
a contradiction.

Suppose that $r_4=1$ and $r_2\ge2$. Then $U=g_1\cdot \ldots \cdot g_lh_3^{r_3-1}h_2^{r_2-1}h_1h_2$ and
\[
\exp(G)+l-1\le\mathsf D(G)=l+r_3+r_2<l-1+ 4\max\{r_2,r_3\}\le\exp(G)+l-1 \,,
\]
a contradiction.   \qed

\smallskip
{\it Proof of \,{\bf A6}}.\,  Assume the contrary  that there is an $i \in [1,l]$ such that $\ord(g_i)>2$, say $i=1$. Then {\bf A4} implies that $g_1=h_1+h_2$. Since $g_1, \ldots, g_l$ are pairwise distinct, it follows that $\ord (g_2)= \ldots = \ord (g_l)=2$ and $k_2= \ldots = k_l=1$. Thus $U = g_1^{k_1}g_2 \cdot \ldots \cdot g_lh_1h_2$, $0=\sigma (U)=k_1g_1+g_2+\ldots + g_l+g_1$, whence $\ord ( (k_1+1)g_1)=2$ and $\ord(g_1)=2(k_1+1)\ge 4$.
It follows that
\[
\exp(G)+l-1\le\mathsf D(G)=|U|=k_1+l+1\le \frac{\ord (g_1)}{2}+l \le \ord(g_1)+l-2< \exp(G)+l-1 \,,
\]
a contradiction.   \qed

\medskip
Now by {\bf A4, A5}, and {\bf A6}, $U$ has the form $U=g_1\cdot \ldots \cdot g_l h_1h_2$ with $\ord(g_i)=2$ for each $i \in [1,l]$.
If $\exp(G)\ge 4$, then
\[
\exp(G)+l-1\le \mathsf D(G)=|U|=l+2<4+l-1 \,,
\]
a contradiction. Thus $G$ must be an elementary $2$-group. Since $(g_1,\ldots,g_l,h_1)$ is a basis of $G$, then $h_2=g_1+\ldots+g_l+h_1$ and $h_1+h_2=h_3+h_4=g_1+\ldots  +g_l$. We can assume that $h_3=h_1+\sum_{i\in I}g_i$ for some $\emptyset \ne I\subsetneq [1,l]$ and $h_4= h_1+\sum_{i\in [1,l]\setminus I}g_i$. Then $A_1=h_3h_1\prod_{i\in I}{g_i}$ and $A_2=h_1h_4\prod_{i\in [1,l]\setminus I}g_i$ are two atoms of lengths $|A_1|, |A_2| \in [3,  \mathsf D (G)-1]$. If $i\in [1,2]$ and $|A_i|\ge4$, then $UVA_i^{-1}$ has to be a product of atoms of length $2$, a contradiction.  Thus it follows that $|A_1|=|A_2|=3$ whence $l=2$ which implies that $\mathsf D (G)=4$, a contradiction.

\medskip
 \noindent
CASE 2.2:\, $U=g_1^{k_1}\cdot \ldots \cdot g_l^{k_l}h^{r-2}h_3^{r_3-1}h_4^{r_4-1}h^2 \ \text{ and}$
\[
  V=(-g_1)^{k_1}\cdot \ldots \cdot (-g_l)^{k_l}(-h)^{r-2}(-h_3)^{r_3-1}(-h_4)^{r_4-1}(-h_3)(-h_4) \,, \quad   \text{where}
\]
$k_1, \ldots, k_l,r_1,\ldots, r_4 \in \mathbb N$, $g_1, \ldots , g_l, h, h_3, h_4$ are pairwise distinct with the only possible exception that $h_3=h_4$ may hold, and $h^2(-h_3)(-h_4) = W_{\mathsf D(G)-1}$.

We start with the following assertions.
\begin{enumerate}
\item[{\bf A7.}\,] For each $i \in [1,l]$  we have  $\ord(g_i)=2$.

\item[{\bf A8.}\,] $h_3=h_4$.

\end{enumerate}

\smallskip
{\it Proof of \,{\bf A7}}.\, Assume to the contrary that there is an $i \in [1,l]$, say $i=1$, such that $\ord (g_1)>2$.

Then, by Lemma \ref{3.4}.2,  there  exists an atom $A \in \mathcal A (G)$ such that   $A\t (-g_1)^{k_1}g_1^{-k_1}U$ and $|A| \in [3,4]$. Since $A\nmid UV(W_{\mathsf D(G)-1})^{-1}$, we must have that $h\t A$.  We distinguish four subcases depending on the multiplicity of $\mathsf v_{-g_1} (A)$ and on $|A|$.

Suppose that $\mathsf v_{-g_1}(A)=3$.  Then $|A|=4$ and $A=(-g_1)^3h$ . It follows that  $A'=g_1^3h(-h_3)(-h_4)$ is a zero-sum subsequence of $UV(A)^{-1}$ which implies that $A'$ is a product of atoms of length $2$, a contradiction.

Suppose that $\mathsf v_{-g_1}(A)=2$.  Then $|A|=4$ by Lemma \ref{3.4}.1.
If $h^2\t A$, then $A=(-g_1)^2h^2$, and hence  $A'=g_1^2(-h_3)(-h_4)$ is an  atom of length $4$ and divides $UVA^{-1}$, a contradiction.
If  $h^2\nmid A$, we obtain that   $A = (-g_1)^2hf$, where  $f \in \{g_2,\ldots,g_l,h_3,h_4\}$ and $(-f)(-h_3)(-h_4)\t V$. Therefore $A'=g_1^2 h(-h_3)(-h_4)(-f)$ is  a zero-sum subsequence of $UVA^{-1}$ which implies that $A'$ is a product of three atoms of length $2$,  a contradiction.

Suppose that $\mathsf v_{-g_1}(A)=1$ and  $|A|=4$.
  If $h^2\t A$, then $A=(-g_1)h^2f$, where   $f \in \{g_2,\ldots,g_l,h,h_3,h_4\}$ and $(-f)(-h_3)(-h_4)\t V$. Hence  $A'=g_1(-h_3)(-h_4)(-f)$ is an  atom of length $4$ and divides  $UVA^{-1}$, a contradiction.
If $h^2\nmid A$, then $A=(-g_1)hff'$ where  $f, f' \in \{g_2,\ldots,g_l,h_3,h_4\}$ and $(-f)(-f')(-h_3)(-h_4)\t V$.  Thus $A'=g_1(-f)(-f')h(-h_3)(-h_4)$ is a  zero-sum subsequence of $UVA^{-1}$ which implies that  $A'$ is a product of three atoms of length $2$,  a contradiction.

Suppose that $\mathsf v_{-g_1}(A)=1$ and  $|A|=3$.
If $h^2\t A$, then $g_1=2h$, a contradiction to Lemma \ref{3.4}.1.
 If $h^2\nmid A$,  then $A=(-g_1)hf$, where $f\in \{g_2,\ldots,g_l,h_3,h_4\}$ and $(-f)(-h_3)(-h_4)\t V$. It follows that  $A'= g_1(-f)h(-h_3)(-h_4)$ is a zero-sum subsequence of $UVA^{-1}$ which implies that $A'$ is a product of two atoms, a contradiction.
 \qed

\smallskip
{\it Proof of \,{\bf A8}}.\, Assume to the contrary that $h_3\neq h_4$.
If $r_3=r_4=1$, then $U=g_1\cdot \ldots \cdot g_l h^r$ with $l\ge 1$, $\ord(h)=2r\ge 4$, and hence
\[
\exp(G)+l-1\le \mathsf D(G)=|U|=l+r\le 2r+l-4<\exp(G)+l-1 \,,
\]
a contradiction. Thus after renumbering if necessary we may assume that  $r_3\ge 2$. We will show this is impossible.

By Lemma \ref{3.4}, there  exists an atom $A \in \mathcal A (G)$ such that $A\t h_3^{-(r_3-1)}(-h_3)^{r_3-1}U$ and $|A|\in[3,4]$. Since $A\nmid UV(W_{\mathsf D(G)-1})^{-1}$, we must have that $h\t A$.  We distinguish four subcases depending on the multiplicity of $\mathsf v_{-h_3} (A)$ and on $|A|$.

Suppose that $\mathsf v_{-h_3}(A)=3$.  Then $|A|=4$ and $A=(-h_3)^3h$ . It follows that  $A'=h_3^2h(-h_4)$ is an atom and divides $UV(A)^{-1}$, a contradiction.

Suppose that $\mathsf v_{-h_3}(A)=2$.  Then $|A|=4$ by Lemma \ref{3.4}.1.
If $h^2\t A$, then $A=(-h_3)^2h^2$ which implies that $h_3=h_4$, a contradiction.
If  $h^2\nmid A$, we obtain that   $A = (-h_3)^2hf$, where  $f \in \{g_1,\ldots,g_l,h_4\}$ and $(-f)(-h_3)(-h_4)\t V$. Therefore $A'=h_3 h(-h_4)(-f)$ is  an atom and divides $UVA^{-1}$,  a contradiction.

Suppose that $\mathsf v_{-h_3}(A)=1$ and  $|A|=4$.
  If $h^2\t A$, then $A=(-h_3)h^2f$, where   $f \in \{g_1,\ldots,g_l,h,h_4\}$ and $(-f)(-h_3)(-h_4)\t V$. Hence $h_4=2h-h_3=-f\t V$, a contradiction.
If $h^2\nmid A$, then $A=(-h_3)hff'$ where  $f, f' \in \{g_1,\ldots,g_l,h_4\}$ and $(-f)(-f')(-h_3)(-h_4)\t V$.  Thus $A'=(-f)(-f')h(-h_4)$ is an atom and divides $UVA^{-1}$,  a contradiction.

Suppose that $\mathsf v_{-h_3}(A)=1$ and  $|A|=3$. Since $h^2(-h_3)(-h_4)$ is an atom, we obtain that  $h^2\nmid A$, and  hence $A=(-h_3)hf$, where $f\in \{g_1,\ldots,g_l,h_4\}$ and $(-f)(-h_3)(-h_4)\t V$. If $f=h_4$, then $h=2h_4$ which implies a contradiction to Lemma \ref{3.4}.1(recall $f=h_4\t U$ and $h=2h_4\t U$). If $f\neq h_4$, by {\bf A7} $\ord(f)=2$ and hence $2h=2h_3$ which implies that $h_3=h_4$, a contradiction. \qed

\medskip
Now, by {\bf A7} and {\bf A8}, $U$ and $V$ have the form
\[
U=g_1\cdot \ldots \cdot g_l h_3^{r_3-2}h^r \quad  \text{and} \quad  V=g_1\cdot \ldots \cdot g_l (-h)^{r-2}(-h_3)^{r_3} \,,
\]
where $l\ge 0$,    $r_3\ge 2$,  $r\ge 2$, $g_1, \ldots , g_l, h_3,h\in G$ are pairwise distinct and $ 2h=2h_3 $.

If $r_3=2$, then $U=g_1\cdot \ldots \cdot g_l h^r$, $\ord(h)=2r$, and
\[
\exp(G)+l-1\le\mathsf D(G)=|U|=l+r\le 2r+l-2 \,,
\]
a contradiction. Considering $V$ and assuming $r=2$ we end again up at a contradiction.
Therefore we obtain that $r_3\ge 3$ and $r \ge 3$.

If $r_3\ge 4$ and $r\ge 4$, then $h^2(-h_3)^2$ and $(-h)^2h_3^2$ are two atoms of length $4$ and divide $UV$, a contradiction.

Thus by symmetry, we may assume that $r_3=3$ and $r\ge 3$.

Suppose that $l=0$. Then $\sigma (U)=h_3+rh=0$ which implies that $G=\langle h \rangle $, a contradiction.

Suppose that $l\ge 1$. Then $2h_3+2rh=0$ which implies that $2(r+1)h=0$.
Thus $\ord(h)=2(r+1)$ or $\ord(h)=r+1$.
If $\ord(h)=2(r+1)$, then
\[
\exp(G)+l-1\le \mathsf D(G)=|U|=l+1+r<l-1+2(r+1)\le \exp(G)+l-1\,,
\]
 a contradiction.
If $\ord(h)=r+1$, then $h=h_3+g_1+\ldots+g_l$ and hence $(-h)h_3g_1\cdot\ldots\cdot g_l$ is an atom and divides $UV(h^2(-h_3)^2)^{-1}$, a contradiction
\end{proof}

\medskip
\begin{proposition} \label{3.7}
Let $G$ be a finite abelian group with $\mathsf D(G) \ge 5$.
Then the following statements are equivalent{\rm \,:}
\begin{enumerate}
\item[(a)] $G$ is either an elementary $2$-group,  or a cyclic group, or isomorphic  to $C_2 \oplus C_{2n}$ with $n \ge 2$.

\smallskip
\item[(b)] There exist  $U, V \in \mathcal A (G)$ with $\mathsf L (UV)= \{2, \mathsf D (G)-1\}$ and $|U|-1=|V|=\mathsf D (G)-1$.
\end{enumerate}
\end{proposition}

\begin{proof}
(a)\, $\Rightarrow$\, (b) \  Suppose that $G$ is an elementary $2$-group and let $(e_1, \ldots, e_r)$ be a basis of $G$ with $\ord (e_1)=\ldots = \ord (e_r)=2$. Then $\mathsf D (G)=r+1$,  $U= e_1 \cdot \ldots \cdot e_{r}e_0 \in \mathcal A (G)$, where $e_0=e_1+ \ldots + e_{r}$,  $V=e_1 \cdot \ldots \cdot e_{r-1}(e_0+e_r) \in \mathcal A (G)$,  and $\mathsf L \big( UV \big) = \{2,r\}$.

Suppose that $G$ is cyclic, and let $e \in G$ with $\ord (e)=|G|=\mathsf D (G)$. Then $U=e^{|G|} \in \mathcal A (G)$, $V=(-e)^{|G|-1}(-2e) \in \mathcal A (G)$, and $\mathsf L \big( UV \big) = \{2, |G|-1\}$.

Suppose that $G$ is isomorphic to $C_2 \oplus C_{2n}$ with $n \ge 2$, and let
$(e_1, e_2)$ be a basis of $G$ with $\ord (e_1)=2$ and  $\ord(e_2)=2n$. Then $\mathsf D (G)=2n+1$, $U = e_1e_2^{2n-1}(e_1+e_2) \in \mathcal A (G)$, $V = (-e_2)^{2n} \in \mathcal A (G)$, and $\mathsf L(UV)=\{2, 2n \}$.

\smallskip
(b)\, $\Rightarrow$\, (a) \ Assume to the contrary that $G$ is neither an elementary $2$-group, nor a cyclic group, nor  isomorphic  to $C_2 \oplus C_{2n}$ for any $n \ge 2$. Then $\mathsf D(G)>\exp(G)+1$.

Let $A\in \mathcal A(G)$ with $A\t UV$. Then we have $|A|\in\{2,3,\mathsf D(G)-1,\mathsf D(G)\}$.
Furthermore, if $|A|=3$, then $UVA^{-1}$ is a product of atoms of length $2$, and if $|A| \in [ \mathsf D(G)-1, \mathsf D(G)]$, then $UVA^{-1} \in \mathcal A (G)$.

Since $\mathsf L (UV)= \{2, \mathsf D (G)-1\}$ and $|U|-1=|V|=\mathsf D (G)-1$,  $U$ and $V$ have the form
\[
U =g_1^{k_1} \cdot \ldots \cdot g_s^{k_s} \quad \text{and} \quad
V =(-g_1-g_2)(-g_1)^{k_1-1}(-g_2)^{k_2-1}(-g_3)^{k_3}\cdot \ldots \cdot (-g_s)^{k_s} \,,
\]
where $s, k_1, \ldots, k_s \in \mathbb N$, $g_1,\ldots, g_s \in G$ are pairwise distinct with the only possible exception that  $g_1=g_2$ may hold.

Since $G$ is not cyclic, we have $s \ge 2$. Suppose that $s=2$.
If $g_1=g_2$, then $\mathsf D(G)=k_1+k_2=\ord(g_1)$ which implies that $G$ is a cyclic group, a contradiction.
Suppose that $g_1\neq g_2$. Since $\mathsf D(G)>\exp(G)+1$,
 Lemma \ref{3.2}.2 implies that  $\mathsf v_{g_1}(U)+\mathsf v_{-g_1}(V)=k_1+k_1-1\le \ord(g_1)$ and $\mathsf v_{g_2}(U)+\mathsf v_{-g_2}(V)=k_2+k_2-1\le \ord(g_2)$. Then
\[
\mathsf D(G)=k_1+k_2\le \frac{\ord(g_1)+1}{2}+\frac{\ord(g_2)+1}{2}\le 1+\exp(G)<\mathsf D(G) \,,
\]
a contradiction.

Thus from now on we suppose that $s \ge 3$, and we continue with the following  assertion.

\begin{enumerate}
\item[{\bf A1.}\,] There exist atoms $U', V' \in \mathcal A (G)$ such that $UV=U'V'$, say
\[
U'=g_1'^{k_1'} \cdot \ldots \cdot g_{s'}'^{k_{s'}'} \quad \text{and} \quad
V'=(-g_1'-g_2')(-g_1')^{k_1'-1}(-g_2')^{k_2'-1}(-g_3')^{k_3'}\cdot \ldots \cdot (-g_{s'}')^{k_{s'}'} \,,
\]
where $s', k_1', \ldots, k_{s'}' \in \mathbb N$, $g_1',\ldots, g_{s'}' \in G$ are pairwise distinct with the only possible exception that  $g_1'=g_2'$ may hold, and $g_1'+g_2'\not\in \supp(U')$.
\end{enumerate}

\noindent
\smallskip
{\it Proof of \,{\bf A1}}.\, If $g_1+g_2\not\in \supp(U)$, then we can choose $U'=U$ and $V'=V$.

Suppose $g_1+g_2\in \supp(U)$, say $g_3=g_1+g_2$. Then $\mathsf v_{-g_3}(V)=k_3+1 \ge 2$ and hence $\ord(g_3)>2$. By Lemma \ref{3.4}.1 and $|U|=\mathsf D(G)$, it follows that $g_1\neq g_2$. We claim that  $k_1=1$ or $k_2=1$.
Indeed, if $k_1\ge 2$ and  $k_2\ge 2$, then $A=g_1g_2(-g_3)$ is an atom  and $A^2\t UV$, a contradiction.

Suppose $s=3$. Without loss of generality, we can assume that $k_1=1$. Since $\mathsf D(G)>\exp(G)+1$,
 Lemma \ref{3.2}.2 implies that  $\mathsf v_{g_2}(U)+\mathsf v_{-g_2}(V)=k_2+k_2-1\le \ord(g_2)$ and $\mathsf v_{g_3}(U)+\mathsf v_{-g_3}(V)=k_3+k_3+1\le \ord(g_3)$. Then
\[
\mathsf D(G)=|U|=1+k_2+k_3\le 1+\frac{\ord(g_2)+1}{2}+\frac{\ord(g_3)-1}{2}\le 1+\exp(G)<\mathsf D(G) \,,
\]
a contradiction.

Thus we obtain that $s\ge4$.
Since
$g_3^{-1}U$ is a zero-sum free sequence of length $\mathsf D(G)-1$, there exists a subsequence $T_1$ of $ g_3^{-1}U$ such that $\sigma(T_1)=(k_3+1)g_3$. Therefore, $(-g_3)^{k_3+1}T_1$ is a zero-sum sequence.  Thus $T_1$ has the form $T_1=g_3^tT_2$ with $t \in [0, k_3-1]$ and $T_2\t g_1^{k_1}g_2^{k_2}g_4^{k_4}\cdot \ldots \cdot g_s^{k_s}$.
Then $(-g_3)^{k_3+1-t}T_2$ is a zero-sum sequence without zero-sum subsequences of length $2$ which implies that $(-g_3)^{k_3+1-t}T_2$ is an atom of length $|(-g_3)^{k_3+1-t}T_2|\in\{3,\mathsf D(G)-1,\mathsf D(G)\}$. Since $k_3+1-t\ge2$, $(-g_3)^{k_3+1-t}T_2$ cannot be an atom of length $3$ by Lemma \ref{3.4}.1, hence $t=0$, $T_1=T_2$, and  $(-g_3)^{k_3+1}T_1$ can only be an atom of length $\mathsf D(G)-1$ or $\mathsf D(G)$. It follows that  $((-g_3)^{k_3+1}T_1)^{-1}UV$ is also an atom and hence $g_4^{k_4}\cdot \ldots \cdot g_s^{k_s}\t T_1$.

Thus any sequence $T$ with $T \t g_3^{-1}U$ and $\sigma (T)=(k_3+1)g_3$ has the property that $g_4^{k_4}\cdot \ldots \cdot g_s^{k_s}\t T$.

We continue with the following two subcases depending on $|(-g_3)^{k_3+1}T_1|$.

 Suppose $|(-g_3)^{k_3+1}T_1|=\mathsf D(G)-1$. Since $((-g_3)^{k_3+1}T_1)^{-1}UV$ is  an atom, we obtain that $k_1=k_2=1$ and  $T_1=(g_1g_2g_3^{k_3})^{-1}U$ which implies that $2(k_3+1)g_3=0$.
Since
$g_4^{-1}U$ is a zero-sum free sequence of length $\mathsf D(G)-1$, there exists a subsequence $W_1$ of $ g_4^{-1}U$ such that $\sigma(W_1)=(k_3+2)g_3$. If $g_3\nmid W_1$, then $g_3^{k_3}W_1$ is a proper zero-sum subsequence of $U$, a contradiction.
If $g_3\t W_1$,  then $g_3^{-1}W_1\t g_3^{-1}U$ and $\sigma (g_3^{-1}W_1)=(k_3+1)g_3$,
but $g_4^{k_4}\cdot \ldots \cdot g_s^{k_s}\nmid g_3^{-1}W_1$, a contradiction.

Suppose $|(-g_3)^{k_3+1}T_1|=\mathsf D(G)$. Since $((-g_3)^{k_3+1}T_1)^{-1}UV$ is an atom, we obtain that $\big(k_1=1$ and $T_1=(g_1g_3^{k_3})^{-1}U\big)$ or  $\big(k_2=1$ and $T_1=(g_2g_3^{k_3})^{-1}U\big)$. By symmetry, we may assume that $k_1=1$, $T_1=(g_1g_3^{k_3})^{-1}U$, and hence $g_1=(-2k_3-1)g_3$. Choose
\[\begin{aligned}
U'&=(-g_3)^{k_3+1}T_1 =g_2(-g_3)g_2^{k_2-1}(-g_3)^{k_3}g_4^{k_4}\cdot \ldots \cdot g_s^{k_s} \quad \text{and} \\
 V'&=((-g_3)^{k_3+1}T_1)^{-1}UV =g_1(-g_2)^{k_2-1}g_3^{k_3}(-g_4)^{k_4}\cdot \ldots \cdot (-g_s)^{k_s} \,,
 \end{aligned}
\]
then $U',V'$ are two atoms with  $U'V'=UV$ and $g_2+(-g_3)=-g_1\not\in \supp (U')$. \qed

\medskip

Thus from now on we may assume that $g_1+g_2\not\in \supp(U)$, and recall that $s \ge 3$. We continue with three further assertions.

\begin{enumerate}
\item[{\bf A2.}\,]  $(\supp(U)+\supp(U))\cap (\supp(U)\setminus \{g_1,g_2\})=\emptyset$.

\item[{\bf A3.}\,]   Let $i\in [3,s]$ with $\ord(g_i)>2$. Then $ \big( -2k_ig_i=g_1$ and $k_1=1 \big)$ or $\big( -2k_ig_i=g_2$ and $k_2=1 \big)$.

\item[{\bf A4.}\,]   If $k_1=1$, then $\big(g_2=-2g_1$ and $k_2=1 \big)$ or $\big(2g_1+2g_2=0$ and $k_2=1 \big)$.
 If $k_2=1$, then $\big(g_1=-2g_2$ and $k_1=1 \big)$ or $\big(2g_1+2g_2=0$ and $k_1=1\big)$.

\end{enumerate}

\smallskip
{\it Proof of \,{\bf A2}}.\, Assume to the contrary that there exists an element $h\in (\supp(U)+\supp(U))\cap (\supp(U)\setminus \{g_1,g_2\})$. Thus  there exist $i, j \in [1, s]$ such that $g_i+g_j = h$ with $h \in \{g_3, \ldots, g_s\}$. Lemma \ref{3.4}.1 implies that $g_i \ne g_j$. Since $A = (-h)g_ig_j$ is an atom of length $3$, then $A^{-1}UV$ is a product of atoms of length $2$. It follows that  $h=g_1+g_2\in \supp(U)$, a contradiction.  \qed

\smallskip
{\it Proof of \,{\bf A3}}.\, By Lemma \ref{3.4}.2, there is an $A \in \mathcal A (G)$ with $A \t g_i^{-k_i}(-g_i)^{k_i}U$ and  $|A| \in \{3,\mathsf D(G)-1\}$. By  {\bf A2} and Lemma \ref{3.4}.1, we obtain that $|A| \ne 3$. Thus $|A|=\mathsf D(G)$. If $A=g_i^{-k_i}(-g_i)^{k_i-1}U$, then $(2k_i-1)g_i=0$ which implies that $\mathsf v_{g_i}(U)+\mathsf v_{-g_i}(V)=2k_i>\ord(g_i)$. Lemma \ref{3.2}.2 implies a contradiction to $\mathsf D(G)>\exp(G)+1$.
Hence
$A=(-g_i)^{k_i}T$ with $T=g_i^{-k_i}h^{-1}U$, where $h\in\supp(U)$. Since $T^{-1}UV$ is an atom, we obtain that  $\big(k_1=1 $ and $h=g_1=-2k_ig_i \big)$ or $\big(k_2=1$ and  $ h=g_2=-2k_ig_i \big)$. \qed

\smallskip
{\it Proof of \,{\bf A4}}.\, Suppose that $k_1=1$. Then $Y=(-g_1-g_2)g_2^{k_2}\cdot \ldots \cdot g_s^{k_s}$ has length $\mathsf D(G)$. By our assumption on $\supp (U)$,  $Y$ has no zero-sum subsequence of length $2$. Therefore,  $Y$ has a zero-sum subsequence $T$ of length $|T|\in \{3, \mathsf D(G)-1,\mathsf D(G)\}$. We continue with the following three subcases depending on $|T|$. If $|T|=3$, say $T = (-g_1-g_2)g_ig_j$ with $i,j \in [2,s]$ not necessarily distinct, then $T$ and  $T' = g_1g_2(-g_i)(-g_j)$ are two atoms and $T'T \t UV$, a contradiction.  If $|T|=\mathsf D(G)$, then $Y=T$ is an atom which implies that $g_2=-2g_1$, $-g_1-g_2=g_1$, and hence $k_1=1$. If  $|T|=\mathsf D(G)-1$, then $k_2=1$, $T=Yg_2^{-1}$, and hence $2g_1+2g_2=0$.

If $k_2=1$, then the assertion follows along the same lines.  \qed

\medskip
The remainder of the proof will be divided into  the following three cases.

\medskip
\noindent CASE 1:\, $|\{i \in [3,s]  \mid \ord(g_i)>2\}|\ge 2$, say $\ord(g_3)>2$ and $\ord(g_4)>2$.

By {\bf A3}, we can assume that $\big(k_1=1$ and  $g_1=-2k_3g_3=-2k_4g_4\big)$ or $\big(g_1=-2k_3g_3, g_2=-2k_4g_4$ and $k_1=k_2=1\big)$.

Consider the sequence $W=g_1^{k_1}g_2^{k_2}(-g_3)^{k_3}(-g_4)^{k_4}g_5^{k_5}\cdot \ldots \cdot g_s^{k_s}$. Since $|W|=\mathsf D(G)$, there exists an atom    $Z \in \mathcal A (G)$ such that $Z \t W$ and $|Z| \in \{3, \mathsf D(G)-1, \mathsf D(G)\}$. We distinguish three subcases depending on $|Z|$.

 Suppose $|Z|=3$. By {\bf A2} and Lemma \ref{3.4}.1, $g_1=g_3+g_4$ or $g_2=g_3+g_4$. If $g_1=g_3+g_4$, then $ (-g_1-g_2)g_3g_4g_2$ and $g_1(-g_3)(-g_4)$ are two  atoms and divide $UV$, a contradiction. If  $g_2=g_3+g_4$, then $(-g_1-g_2)g_1g_3g_4$ and $g_2(-g_3)(-g_4)$ are two  atoms and divide $UV$, a contradiction.

Suppose $|Z|=\mathsf D(G)-1$. Since $UVZ^{-1}$ is a atom, we obtain that $Z=Wg_1^{-1}$ or $Z=Wg_2^{-1}$. Therefore, $-2k_3g_3-2k_4g_4=g_1$ or $g_2$. Our assumption infers that $-2k_3g_3-2k_4g_4=g_2$ and $g_1=-2k_3g_3=-2k_4g_4$ which implies that $2g_1=g_2$, a contradiction to Lemma \ref{3.4}.1.

Suppose $|Z|=\mathsf D(G)$. Then we obtain that $2k_3g_3+2k_4g_4=0$.  Our assumption infers that $k_1=1$, $g_1=-2k_3g_3=-2k_4g_4$, and hence  $\ord(g_1)=2$.
 Therefore, $4k_3g_3=0$, $g_1=2k_3g_3$, and hence $k_3\ge 2$ by Lemma \ref{3.4}.1. Since $g_1^{-1}U$ is a zero-sum sequence of length $\mathsf D(G)-1$, there exists a subsequence $W$ of $g_1^{-1}U$ such that $\sigma(W)=(2k_3+1)g_3$. If $g_3\t W$, then $g_1g_3^{-1}W$ is a proper zero-sum subsequence of $U$, a contradiction. Suppose $g_3\nmid W$. Then $g_1(-g_3)W$ is an atom and divides $UV$ which implies that $|g_1(-g_3)W|\in \{3,\mathsf D(G)-1, \mathsf D(G)\}$. Since $g_3(-g_3)\t UV(g_1(-g_3)W)^{-1}$, then $UV(g_1(-g_3)W)^{-1}$ is not an atom. Thus $|g_1(-g_3)W|=3$ which implies that $g_3\in \supp (U)\cap\supp (V)$,  a contradiction to {\bf A2}.

\medskip
\noindent CASE 2:\, $|\{i \in [3,s]  \mid \ord(g_i)>2\}|=1$, say $\ord(g_3)>2$.

By {\bf A3}, we may assume  that $k_1=1$ and $g_1=-2k_3g_3$.
By {\bf A4}, we obtain that $\big(g_2=-2g_1$ and $k_2=1 \big)$ or $\big(2g_1+2g_2=0$ and $k_2=1 \big)$. We continue with the following two subcases.

Suppose that $2g_1+2g_2=0$ and  $k_2=1$.
Since $\sigma(U)=g_1+g_2+k_3g_3+g_4+\ldots +g_s=0$, we obtain that $2k_3g_3=0=-g_1$, a contradiction.

Suppose that  $g_2=-2g_1=4k_3g_3$ and $k_2=1$. If $s=3$, then $G=\langle g_3 \rangle$ is cyclic, a contradiction. Hence $s\ge 4$ and $G=\langle g_1,\ldots ,g_{s-1}\rangle=\langle g_3,\ldots ,g_{s-1}\rangle$ which implies that $\mathsf r(G)=s-3$ and  $\exp(G)=\ord(g_3)$ is even.  Since  $\mathsf D(G)>\exp(G)+1$, by Lemma \ref{3.2}.2, we infer that
$\mathsf v_{g_3}(U)+\mathsf v_{-g_3}(V)=2k_3\le \ord(g_3)$.
Therefore, $g_1=-2k_3g_3\neq 0$ infers that $\ord(g_3)\ge 2k_3+2\ge4$.
Thus \[
\exp(G)+s-4\le \mathsf D(G)=|U|=k_3+s-1\le \frac{\ord(g_3)}{2}-1+s-1\le
\ord(g_3)+s-4
\]
which implies that $\ord(g_3)=4$,
a contradiction to  $g_2=4k_3g_3\neq 0$.

\medskip
\noindent CASE 3:\, $|\{i \in [3,s]  \mid \ord(g_i)>2\}|=0$.

 Since $\sigma (U)=k_1g_1+k_2g_2+g_3+\ldots+g_s=0$, we obtain that $2k_1g_1+2k_2g_2=0$.

Suppose that $g_1=g_2$. Then $\ord(g_1)=2(k_1+k_2)\ge 4$. It follows that
\[
\mathsf D(G)=k_1+k_2+s-2=\frac{\ord(g_1)}{2}+s-2<\ord(g_1)-1+s-2\le \mathsf D(G)\,,
\]
a contradiction.

Thus $g_1\neq g_2$. Consider the sequence $S=(-g_1-g_2)(-g_1)^{k_1-1}g_2^{k_2}g_3\cdot \ldots \cdot g_s$.  Since $|S|=\mathsf D(G)$, there exists an atom $Z \in \mathcal A (G)$ such that   $Z  \t S$ and $|Z| \in \{3, \mathsf D(G)-1, \mathsf D(G)\}$. We distinguish three subcases depending on $|Z|$.

Suppose that $|Z|=\mathsf D(G)$. Then $Z=S$ and  $g_2=-2k_1g_1=2k_2g_2$. Hence $(2k_2-1)g_2=0$ and $\ord(g_2)=2k_2-1$. If $s=3$, then $G=\langle g_1 \rangle$ is cyclic, a contradiction. Hence $s\ge 4$ and  $G=\langle g_1,\ldots ,g_{s-1}\rangle=\langle g_1,g_3,\ldots ,g_{s-1}\rangle$ which implies that  $\mathsf r(G)= s-2$ and $\ord(g_1)=\exp(G)$ is even.
Then $\ord(g_1)\ge 2\ord(g_2)\ge 6$.  Since  $\mathsf D(G)>\exp(G)+1$, by Lemma \ref{3.2}.2, we infer that
$\mathsf v_{g_1}(U)+\mathsf v_{-g_1}(V)=2k_1-1\le \ord(g_1)$ which implies that $2k_1\le \ord(g_1)$. Since $g_2=-2k_1g_1\ne 0$, we obtain that $2k_1\le \ord(g_1)-2$.
Thus
\[
\begin{aligned}
 \exp(G)+s-3&\le \mathsf D(G)=|U|=k_1+k_2+s-2\le \frac{\ord(g_1)}{2}-1+\lfloor\frac{\ord(g_1)+2}{4}\rfloor+s-2\\
 &\le \frac{\ord(g_1)}{2}-1+\frac{\ord(g_1)}{2}-1+s-2\le\ord(g_1)+s-4
 \end{aligned}
\] a contradiction.

Suppose that  $|Z|=\mathsf D(G)-1$. Then there exists an element $h\t S$ such that $Z=h^{-1}S$. Since $Z^{-1}UV$ is an atom, we obtain that $\big(h=-g_1-g_2\big)$ or $\big(h=g_2$ and $k_2=1\big)$. If $h=-g_1-g_2$,
then $Z^{-1}UV=(-g_1-g_2)g_1^{k_1}(-g_2)^{k_2-1}g_3\cdot \ldots \cdot g_s$ is an atom of length $\mathsf D(G)$ and hence the similar argument of the previous subcase $|Z|=\mathsf D(G)$ implies a contradiction. Suppose that $h=g_2$ and $k_2=1$. By {\bf A4}, we obtain that $k_1=k_2=1$. Thus $\exp(G)+s-3\le \mathsf D(G)=|U|=s$ which implies that $\exp(G)\le 3$, a contradiction.

Suppose that $|Z|=3$.  Then $UV Z^{-1}$ can only be a product of atoms of length $2$. But $\mathsf v_{g_1}(UV Z^{-1})=k_1 > \mathsf v_{-g_1}( UV)\ge \mathsf v_{-g_1}( UV Z^{-1})$, a contradiction.
\end{proof}

\medskip
\begin{proposition} \label{3.8}
Let $G$ be a finite abelian group with $\mathsf D (G) \ge  5$.
Then the following statements are equivalent{\rm \,:}
\begin{enumerate}
\item[(a)] $G$ is either an elementary $2$-group, or  isomorphic  to $C_2^{r-1}\oplus C_4$ for some  $r \ge 2$, or isomorphic to $C_2 \oplus C_{2n}$ for some $n \ge 2$.

\smallskip
\item[(b)] There exist  $U, V \in \mathcal A (G)$ with $\mathsf L (UV)= \{2, \mathsf D (G)-1\}$ and $|U|=|V|=\mathsf D (G)-1$.
\end{enumerate}
\end{proposition}

\begin{proof}
(a)\, $\Rightarrow$\, (b) \ Suppose that $G$ is an elementary $2$-group, and let $(e_1, \ldots, e_r)$ be a basis of $G$ with $\ord (e_1)=\ldots = \ord (e_r)=2$. Then $\mathsf D (G)=r+1$, $U= e_1 \cdot \ldots \cdot e_{r-1}e_0 \in \mathcal A (G)$, where $e_0=e_1+ \ldots + e_{r-1}$, and $\mathsf L \big( U(-U) \big) = \{2,r\}$.

Suppose that $G$ is isomorphic to $C_2^{r-1}\oplus C_4$ for some  $r \ge 2$, and let  $(e_1, \ldots, e_r)$ be a basis of $G$ with $\ord (e_1)=\ldots = \ord (e_{r-1})=2$ and $\ord (e_r)=4$.  Then $\mathsf D (G)=r+3$, $U=e_1\cdot \ldots \cdot e_{r-1}e_r^2 (e_0+e_r) \in \mathcal A (G)$, where $e_0=e_1+ \ldots + e_r$, and $\mathsf L \big( U(-U) \big) = \{2, r+2\}$.

Suppose that $G$ is isomorphic to $C_2 \oplus C_{2n}$ for some $n \ge 2$, and let
$(e_1, e_2)$ be a basis of $G$ with $\ord (e_1)=2$ and  $\ord(e_2)=2n$. Then $\mathsf D (G)=2n+1$, $U = e_2^{2n} \in \mathcal A (G)$,  and $\mathsf L \big( U(-U) \big)=\{2, 2n \}$.

\smallskip
(b)\, $\Rightarrow$\, (a) \ Clearly, we have $V=-U$, and
for every zero-sum sequence $W$ with $W\t UV$ and  $W\neq UV$, it follows that $W$ is either an atom of length $\mathsf D(G)-1$ or $W$ is a product of atoms of length $2$.
We set  $U=g_1^{k_1}\cdot \ldots \cdot g_l^{k_l}$ with $l, k_1, \ldots, k_l \in \mathbb N$ and $g_1, \ldots, g_l \in G$ pairwise distinct.

Suppose that $l=1$. Then $k_1=\ord (g_1)=\mathsf D (G)-1$. Thus  $k_1$ is even and $G \cong C_2 \oplus C_{k_1}$.

Suppose that $l \ge 2$.
For each $i\in [1,l]$,  the sequence $S_i=g_i^{-k_i}(-g_i)^{k_i}U$ is either zero-sum free or an atom; clearly, $S_i$ is an atom if and only if  $2k_ig_i=0$.
So we can distinguish  two cases.

\smallskip
\noindent CASE 1:\, For each $i\in [1,l]$ we have $2k_ig_i=0$.

We claim that for any $i\in [1,l]$, the tuple $(g_1,\ldots ,g_{i-1}, g_{i+1}, \ldots , g_l)$ is independent. Clearly, it is sufficient to prove the claim for $i=l$. Assume to the contrary that $(g_1, \ldots, g_{l-1})$ is not independent.
Then there is an atom $W$ with $W\t g_1^{k_1}\cdot \ldots \cdot g_{l-1}^{k_{l-1}}\cdot (-g_1)^{k_1}\cdot \ldots \cdot (-g_{l-1})^{k_{l-1}} $ and $|W|>2$. Then $W$ is an atom of length $\mathsf D (G)-1$, and thus $W^{-1}UV$ is also an atom of length $\mathsf D (G)-1$. But  $g_l(-g_l)\t W^{-1}UV$, a contradiction.
After renumbering if necessary we may suppose that
\[
\ord(g_l)=\min \{\ord(g_1),\ldots , \ord(g_l)\} \quad \text{and} \quad \ord(g_1)=\min \{\ord(g_1),\ldots , \ord(g_{l-1})\} \,.
\]
Suppose that $l=2$. By our assumption that $\ord (g_1) \ge \ord (g_2)$, we obtain that
\[
\begin{aligned}
\exp (G)-1 & \le \mathsf D (G)-1=k_1+k_2=\frac{\ord (g_1)}{2}+\frac{\ord (g_2)}{2}=\ord (g_1)-\frac{\ord (g_1)-\ord(g_2)}{2}\\ & \le \exp (G)- \frac{\ord (g_1)-\ord(g_2)}{2} \le \exp (G) \,.
\end{aligned}
\]
If $\mathsf D (G)= \exp (G)$, then $\ord (g_1)=\exp (G)$, $\ord (g_1)-\ord (g_2)=2$, and hence $\ord (g_2)\t 2$ which implies $\ord (g_2)=2$ and $\ord (g_1)=4=\mathsf D (G)=\exp (G)$, a contradiction to $\mathsf D (G)\ge 5$.
If $\mathsf D (G)=\exp (G)+1$, then $G\cong C_2 \oplus C_{2n}$ for some $n \ge 2$.

Suppose that $l \ge 3$. Then
\[
\begin{aligned}
\mathsf D(G)-1 & = |U| = k_1+\ldots +k_l=\frac{\ord(g_1)}{2}+\ldots +\frac{\ord(g_{l})}{2}\le\ord(g_1)+\frac{\ord(g_2)}{2}+\ldots +\frac{\ord(g_{l-1})}{2} \\ & \le \ord(g_1)+\ldots  +\ord(g_{l-1})-(l-2)\le \mathsf D(G) \,.
\end{aligned}
\]

Suppose that equality holds at the second inequality sign. Then
\[
\ord(g_1)+\frac{\ord(g_2)}{2}+\ldots +\frac{\ord(g_{l-1})}{2}= \ord(g_1)+\ldots  +\ord(g_{l-1})-(l-2) \,.
\]
Since $\ord (g_i)/2 \le \ord (g_i)-1$ for all $i \in [2, l-1]$, it follows that then $\ord(g_i)=2$ for all $i \in [2,l-1]$. Because our assumption on the order of the elements, we infer that $\ord (g_1)=\ord (g_l)=2$, and hence $G$ is an elementary $2$-group.

Suppose that  inequality holds at the second inequality sign. Then we have
\[
\ord(g_1)+\frac{\ord(g_2)}{2}+\ldots +\frac{\ord(g_{l-1})}{2}= \ord(g_1)+\ldots  +\ord(g_{l-1})-(l-2)-1 \,.
\]
Since $\ord (g_i)/2 \le \ord (g_i)-1$ for all $i \in [2, l-1]$, there exists an $i \in [2, l-1]$, say $i=2$, such that $\ord (g_2)=4$ and $\ord (g_i)=2$ for all $i \in [3, l-1]$.
If $l=3$, then $G \cong C_4 \oplus C_4$ or $G \cong C_2  \oplus C_4$, but the first case is a contradiction to Lemma \ref{3.3}. If $l \ge 4$, then $G\cong C_2^{l-2}\oplus C_4$.

\smallskip
\noindent CASE 2:\, There exists an  $i\in [1,l]$ such that the sequence $S_i$ is zero-sum free, say $i=1$.

We start with a list of assertions.

\begin{enumerate}
\item[{\bf A1.}\,] Let $T, T' \in \mathcal F (G)$ be distinct such that $T \t U$, $T' \t U$, and $\sigma (T)=\sigma (T')$. Then $\supp(T)\cap \supp(T')=\emptyset$, $T T' = U$, and $2\sigma(T)=0$.

\item[{\bf A2.}\,] Let $T, T' \in \mathcal F (G)$ be distinct such that $T \t S_1$, $T' \t S_1$, and $\sigma (T)=\sigma (T')$. Then $\supp(T)\cap \supp(T')=\emptyset$, $T T' = S_1$, and $2\sigma(T)= - 2k_1g_1$.

\item[{\bf A3.}\,] $\Sigma (U) = G$.

\item[{\bf A4.}\,] If $i \in [1,l]$ with $\ord (g_i)>2$, then $\ord (g_i)>2k_i$.

\end{enumerate}

\smallskip
{\it Proof of \,{\bf A1}}.\, Obviously,  $T(-T')$ has sum zero and $T(-T')\t UV$  but $T(-T')\neq UV$. So $T(-T')$ must be an atom of length $\mathsf D(G)-1$ which implies that $\supp(T)\cap \supp(T')=\emptyset$, $T T' = U$, and $2\sigma(T)=0$. \qed

\smallskip
{\it Proof of \,{\bf A2}}.\, Obviously,  $T(-T')$ has sum zero, $T(-T') \t UV$, and $T(-T')\neq UV$. So $T (-T')$ must be an atom of length $\mathsf D(G)-1$ which implies that $\supp(T)\cap \supp(T')=\emptyset$, $T T' = S_1$, and $2\sigma(T)=-2k_1g_1$. \qed

\smallskip
{\it Proof of \,{\bf A3}}.\,
We will show that  $|\Sigma (U)| = |G|$. Clearly, we have
\[
\begin{aligned}
|\Sigma (U)| & = |\{ \sigma (T) \mid 1 \ne T \in \mathcal F (G), \ T \t U \}|  \\
& = |\{g \in G \mid \text{there exist} \ 1 \ne T \ \text{with} \ T\t U \ \text{and} \ \sigma (T)=g \}| \end{aligned}
\]
Since $U = g_1^{k_1} \cdot \ldots \cdot g_l^{k_l}$, we have
\[
|\{ T \in \mathcal F (G) \mid 1 \ne T, \ T \t U \}| = (k_1+1)\cdot \ldots \cdot(k_l+1)-1 \,.
\]
By {\bf A1},  there are at most two distinct subsequences of $U$ with given sum  $g \in G$.
Therefore we obtain
\[
\begin{aligned}
|\Sigma (U)| & = |\{ \sigma (T) \mid 1 \ne T \in \mathcal F (G), \ T \t U \}|  \\
& = |\{ T \in \mathcal F (G) \mid 1 \ne T, \ T \t U \}| - \\
 & \qquad \frac{1}{2}|\{ T \in \mathcal F (G) \mid \ T\t U \ \text{and there is a divisor $T'$ of $U$ with $T\ne T'$ and $\sigma (T)=\sigma (T')$} \}| \\
&=(k_1+1)\cdot \ldots \cdot(k_l+1)-1-|\{T \in \mathcal F (G)  \ \mid \  g_1\nmid T, \ T=\prod\limits_{g\in \supp(T)}g^{\mathsf v_{g}(U)}, \ \text{ and }  \ord(\sigma(T))=2  \}| \,.
\end{aligned}
\]
Next we study $|\Sigma (S_1)|$. Since $S_1$ is zero-sum free of length $|S_1|=\mathsf D (G)-1$, it follows that $\Sigma (S_1)=G\setminus \{0\}$.
Using {\bf A2} for the second equality sign we obtain that
\[
\begin{aligned}
|G|-1 & = |\Sigma (S_1)| = (k_1+1)\cdot \ldots \cdot(k_l+1)-1\\
& \qquad -|\{T \in \mathcal F (G) \ \mid \   T=(-g_1)^{k_1} \prod\limits_{g\in \supp(T)\setminus\{-g_1\}}g^{\mathsf v_{g}(U)}\text{ and } 2\sigma(T)=-2k_1g_1  \}|\\
&=  (k_1+1)\cdot \ldots \cdot(k_l+1)-1\\
&\qquad -\big(|\{T\in \mathcal F(G) \ \mid \  g_1\nmid T \text{ and } T=\prod\limits_{g\in \supp(T)}g^{\mathsf v_{g}(U)},\ \ord(\sigma(T))=2  \}|+|\{T\in \mathcal F(G)\mid T=1\}|\big)\\
&=  |\Sigma (U)|-1 \,,
\end{aligned}
\]
and hence $\Sigma (U)=G$. \qed

\smallskip
{\it Proof of \,{\bf A4}}.\, Since $S_1$ is zero-sum free, it follows  that $\ord (g_1)>2k_1$.
Now let $i \in [2,l]$ be given and   assume to the contrary that $\ord (g_i)\le 2k_i$. Recall that $\ord (g_i) > \mathsf v_{g_i}(U)=k_i$.
We set $U = g_i^{k_i}U'$. Then $W'=(-g_i)^{\ord (g_i)-k_i}U'$ has sum zero and divides $UV$. Thus $W'$ is an atom of length $\mathsf D (G)-1 =|W'|=|U|$ and hence $\ord (g_i)=2k_i$. Since $\Sigma (S_1)=G\setminus \{0\}$, $S_1$ has a subsequence $T$ such that $\sigma (T)=(k_i+1)g_i$. If $g_i \nmid T$, then $Tg_i^{k_i-1}$ is a proper zero-sum subsequence of $S_1$, a contradiction. If $g_i \t T$, then $\sigma (g_i^{-1}T)=k_ig_i = \sigma (g_i^{k_i})$. By {\bf A2}, it follows that $0= 2k_ig_i= 2 \sigma (g_i^{-1}T )= -2k_1g_1 \ne 0$, a contradiction. \qed

\medskip
Now we distinguish two subcases.

\smallskip
\noindent
CASE 2.1:\, $|\{ i \in [1,l] \mid \ord(g_i)>2\}|\ge 3$, say $\ord(g_1)>2$, $\ord(g_2)>2$, and  $\ord(g_3)>2$.

 We start with the following assertion.

 \begin{enumerate}
 \item[{\bf A5.}\,]  There is a subsequence $W$ of $U$ with $\sigma (W)= g_1-g_2$ such that  $W=g_2^{k_2}W'$ and for any $h\t W'$, $\ord (h)=2$.
 \end{enumerate}

\smallskip
{\it Proof of \,{\bf A5}}.\,
By {\bf A3},  there exists some $W \in \mathcal F (G)$ such that $W\t U$ and $\sigma(W)=g_1-g_2$. We claim that $g_1\nmid W$ but $g_2^{k_2}\t W$. Assume to the contrary that $g_1\t W$. Then $\sigma(g_1^{-1}W)=-g_2=\sigma (g_2^{-1}U)$. By {\bf A1}, this implies that $2g_2=0$, a contradiction.
Assume to the contrary that $g_2^{k_2}\nmid W$.  Then $g_2W\t U$ and $\sigma (g_2W)=g_1=\sigma(g_1)$.  By {\bf A1}, this implies that $2g_1=0$, a contradiction.
Thus  $W=g_2^{k_2}W'$ with $W'\t g_3^{k_3}\cdot \ldots \cdot g_l^{k_l}$.

 Let $i \in [3, l]$ such that  $g_i\t W'$. We only need to  show that $\ord(g_i)=2$.
Assume to the contrary that $\ord (g_i)>2$. We set $X=Ug_ig_2^{-1}$, and then $|X|=|U|=\mathsf D(G)-1$. Suppose that $X$ has a zero-sum subsequence $T$. Then $g_i^{k_i+1}\t T$ and by {\bf A4} we obtain that   $\ord(g_i)>2k_i$.   Therefore, $g_i^{-k_i}T\t U$ and $g_i^{-k_i}T \neq g_i^{-k_i}U$ but  $\sigma( g_i^{-k_i}T)=-k_ig_i=\sigma (g_i^{-k_i}U)$ which implies that $2k_ig_i=0$ by {\bf A1}, a contradiction.
Thus $X$ is zero-sum free, and $|X|=\mathsf D (G)-1$ which imply that $\Sigma (X)=G\setminus \{0\}$.
Therefore, $X$ has a subsequence $T$ such that $\sigma (T)=g_1-g_2$.

 Suppose that  $g_i^{k_i+1}\nmid T$. Then $T\t U$, and by definition of $X$ we have  $g_2^{k_2}\nmid T$ which implies that  $g_2T \t U$. Since  $\sigma (g_2T)=g_1= \sigma (g_1)$, we obtain that  $2g_1=0$ by {\bf A1} , a contradiction.

 Suppose that  $g_i^{k_i+1}\t T$.  Then $g_i^{-1}T\t U$.
Since $\sigma (g_i^{-1}T)=g_1-g_2-g_i=\sigma(g_i^{-1}W)$ and $g_i^{-1}T\neq g_i^{-1}W$, we obtain that  $\supp(g_i^{-1}T)\cap \supp(g_i^{-1}W)=\emptyset$ and $U=g_i^{-1}T\cdot g_i^{-1}W$ by {\bf A1}. Set $T_1=g_i^{-1}T$,  then $g_1^{k_1}g_i^{k_i}\t T_1$ and $g_2\nmid T_1$. It follows that $\sigma(g_1^{-1}g_2T_1)=-g_i=\sigma (g_i^{-1}U)$ which implies that $2g_i=0$ by {\bf A1}, a contradiction.        \qed

\smallskip
Repeating the argument of {\bf A5}, we can find another subsequence $W_1$ of  $U$ with $\sigma(W_1)=g_2-g_1$ such that  $W_1=g_1^{k_1}W_1'$ and for any $h\t W_1'$, $\ord (h)=2$. Set $Y=WW_1(\prod\limits_{h\in \supp(W')\cap \supp(W_1')}h^2)^{-1}$, then $Y$ is a zero-sum subsequence of $U$.  Since $\ord(g_3)>2$, we have $g_3\nmid W$ and $g_3\nmid W_1$ which imply that $g_3\nmid Y$. It follows that $Y$ is a proper zero-sum subsequence of $U$, a contradiction.

\medskip
\noindent
CASE 2.2:\,  $|\{i \in [1,l]  \mid \ord(g_i)>2\}|\le 2$.

Since $k_1g_1+ \ldots + k_lg_l=\sigma (U)=0$ and $\ord (k_1g_1)> 2$, it follows that $|\{i \in [1,l]  \mid \ord(g_i)>2\}| = 2$, say, $\ord(g_1)>2$ and $\ord(g_2)>2$. Then {\bf A4} implies that $\ord(g_1)>2k_1$ and $\ord(g_2)>2k_2$.

Suppose that  $l=2$. Then \[
\mathsf D(G)-1=k_1+k_2\le \frac{\ord(g_1)-1}{2}+\frac{\ord(g_2)-1}{2}\le \exp(G)-1\] which implies that $G$ is a cyclic group and $k_1=k_2$. Since any minimal zero-sum sequence of length $|G|-1$ over a cyclic group has the form $g^{|G|-2}(2g)$ for some generating element $g \in G$, it follows that $1=k_2=k_1$, and hence $|G|=3$, a contradiction to $\mathsf D (G)\ge 5$.

Suppose  $l\ge 3$. Then $\exp(G)$ is even.  We may assume that $\ord(g_1)\ge \ord(g_2)$. Since $(g_3,\ldots,g_l)$ is independent, we have $\mathsf r(G)\ge l-2$. Therefore, \[
\begin{aligned}\exp(G)+l-4&\le \mathsf D(G)-1=|U|=k_1+k_2+l-2\\
&\le \lfloor\frac{\ord(g_1)-1}{2}\rfloor+\lfloor\frac{\ord(g_2)-1}{2}\rfloor+l-2
\le 2\lfloor\frac{\ord(g_1)-1}{2}\rfloor+l-2\\
 &\le \ord(g_1)+l-3
 \le \exp(G)+l-3\,.
\end{aligned}
\]
Since $\ord(g_1)\t \exp(G)$, we have that $\ord(g_1)=\exp(G)$ is even.
Thus $2\lfloor\frac{\ord(g_1)-1}{2}\rfloor=\ord(g_1)-2$ which implies that   $k_1+k_2=\lfloor\frac{\ord(g_1)-1}{2}\rfloor+\lfloor\frac{\ord(g_2)-1}{2}\rfloor= 2\lfloor\frac{\ord(g_1)-1}{2}\rfloor$. Then $\ord(g_1)=\ord(g_2)=2k_1+2=2k_2+2$. Since $\sigma(U)=k_1g_1+k_2g_2+g_3+\ldots+g_l=0$, we have  $2k_1g_1+2k_2g_2=0$ which implies that $2g_1+2g_2=0$. If $k_1=k_2\ge 2$, then $g_1^2g_2^2$ is an atom and divides $UV$, a contradiction. Thus $k_1=k_2=1$, and hence $\ord(g_1)=\ord(g_2)=4$. By {\bf A3}, there exists a subsequence  $W$ of $U$ such that $\sigma (W)=2g_1$. If $g_1\t W$, then $g_1\nmid g_1^{-1}W$ and $\sigma (g_1^{-1}W)=g_1=\sigma(g_1)$ which implies that $2g_1=0$ by {\bf A1}, a contradiction.
If $g_1\nmid W$, then $g_1\nmid U(g_1W)^{-1}$ and $\sigma (U(g_1W)^{-1})=g_1=\sigma(g_1)$ which implies that $2g_1=0$ by {\bf A1}, a contradiction.
\end{proof}

\medskip
\section{Proof of the Main Results} \label{4}
\medskip

In this final section we provide the proofs of all results presented in the Introduction (Theorem \ref{1.1}, Corollary \ref{1.2}, and Corollary \ref{1.3}).

\medskip
\begin{proof}[Proof of Theorem \ref{1.1} and of Corollary \ref{1.2}]
Let $H$ be a Krull monoid with finite class group $G$ such that $|G| \ge 3$ and  each class contains a prime divisor. Recall that the monoid of zero-sum sequences $\mathcal B (G)$ is a Krull monoid with class group isomorphic to $G$ and each class contains a prime divisor. By Proposition \ref{2.1}, $\daleth (H) = \daleth (G)$ and $\mathsf c (H) = \mathsf c  (G)$. Thus it is sufficient to prove Theorem \ref{1.1} for the Krull monoid $\mathcal B (G)$.

Let $\mathcal O$ be a holomorphy ring in a global field $K$,  and $R$ a classical maximal $\mathcal O$-order in a central simple algebra $A$ over $K$ such that every stably free left $R$-ideal is free. Then the monoid $R^{\bullet}$ is a non-commutative Krull monoid (\cite{Ge13a}), and all invariants under consideration of $R^{\bullet}$ coincide with the respective invariants of a commutative Krull monoid whose class group is isomorphic to a ray class group of $\mathcal  O$. These (highly non-trivial) transfer results are established in \cite{Sm13a, Ba-Sm15}, and  are summarized in \cite[Theorems 7.6 and 7.12]{Ba-Sm15}.
Therefore, both for Theorem \ref{1.1} and for Corollary \ref{1.2}, it is sufficient to prove the equivalence of the statements for a monoid of zero-sum sequences.

Let $G$ be a finite abelian group with $|G| \ge 3$, and recall the inequalities
\[
\daleth (G)  \le \mathsf c (G) \le \mathsf D (G) \,.
\]
(c)\, $\Rightarrow$\, (a) \ Suppose that $G$ is isomorphic either to $C_2^{r-1} \oplus C_4$ for some $r \ge 2$ or  to $C_2 \oplus C_{2n}$ for some $n \ge 2$.
Then Theorem {\bf A} (in the Introduction)  shows that $\mathsf c (G) \le \mathsf D (G)-1$.
Since $\mathsf D (C_2^{r-1}\oplus C_4)= r+3$ and  $\mathsf D (C_2 \oplus C_{2n})=2n+1$, Lemma \ref{3.1}.2 implies that $\mathsf c (G) \ge \daleth (G) \ge \mathsf D (G)-1$.

(a)\, $\Rightarrow$\, (b) \ Suppose that $\mathsf c (G) = \mathsf D (G)-1$. By Theorem {\bf A},  $G$ is neither cyclic nor an elementary $2$-group which implies that $\mathsf D(G)\ge 5$.
By Lemma \ref{3.1}.1, we have
\[
\mathsf D (G)-1 = \mathsf c (G)   \le  \max \Big\{ \Big\lfloor\frac{1}{2} \mathsf D(G)+1 \Big\rfloor,\, \daleth (G) \Big\}   \le \mathsf D(G) \,.
\]
Thus,  if $\daleth (G)<\mathsf D (G)-1$, then $\Big\lfloor\frac{1}{2} \mathsf D(G)+1 \Big\rfloor \ge \mathsf D(G)-1$ which implies that $\mathsf D(G)\le 4$, a contradiction.

(b)\, $\Rightarrow$\, (c) Suppose that $\daleth (G) = \mathsf D (G)-1$. Again by Theorem {\bf A}, $G$ is neither cyclic nor an elementary $2$-group which implies that $\mathsf D (G)\ge 5$. Then there exist $U, V \in \mathcal A (G)$ such that $\min \big( \mathsf L (UV) \setminus \{2\} \big) = \mathsf D (G)-1$. Obviously, there are the following four cases (up to symmetry):
\begin{itemize}
\item $\mathsf L (UV)= \{2, \mathsf D (G)-1, \mathsf D (G)\}$.

\item $\mathsf L (UV)=\{2, \mathsf D (G)-1\}$ and $|U|=|V|=\mathsf D (G)$.

\item $\mathsf L (UV)=\{2, \mathsf D (G)-1\}$ and $|U|-1=|V|=\mathsf D (G)$.

\item $\mathsf L (UV)=\{2, \mathsf D (G)-1\}$ and $|U|=|V|=\mathsf D (G)-1$.
\end{itemize}
These cases are handled in the Propositions \ref{3.5} to \ref{3.8}, and they imply that $G$ is isomorphic either to $C_2^{r-1} \oplus C_4$ for some $r \ge 2$ or  to $C_2 \oplus C_{2n}$ for some $n \ge 2$.
\end{proof}

\medskip
\begin{proof}[Proof of  Corollary \ref{1.3}]
Let $G$ and $G'$ be abelian groups such that $\mathcal L (G) = \mathcal L (G')$. Then
\[
\daleth (G) = \sup \{ \min ( L \setminus \{2\}) \mid 2 \in L \in \mathcal L (G)\} = \daleth (G') \,.
\]
If $G$ is finite, then $\daleth (G) \le \mathsf c (G) \le \mathsf D (G) < \infty$. If $G$ is infinite, then, by the Theorem of Kainrath (see \cite{Ka99a} or \cite[Section 7.3]{Ge-HK06a}), every finite set $L \subset \mathbb N_{\ge 2}$ lies in $\mathcal L (G)$, which implies that $\daleth (G)=\infty$.

For $k \in \mathbb N$, we define the refined elasticities
\[
\rho_k (G) = \sup \{ \sup L \mid k \in L \in \mathcal L (G) \} \,,
\]
and observe  that $\rho_k (G) = \rho_k (G')$. This implies that $k \mathsf D (G) = \rho_{2k} (G)=\rho_{2k} (G') = k \mathsf D (G')$ (see \cite[Section 6.3]{Ge-HK06a}) for each $k \in \mathbb N$, and hence $\mathsf D (G) = \mathsf D (G')$.

Now suppose that $G' \in \{C_2^{r-1} \oplus C_4, C_2 \oplus C_{2n} \}$ where $r, n \ge 2$. Then Theorem \ref{1.1} implies that $\daleth (G') = \mathsf D (G')-1$. Since $\mathcal L (G) = \mathcal L (G')$, it follows that $G$ is finite and that
\[
\daleth (G) = \daleth (G') = \mathsf D (G')-1 = \mathsf D (G)-1 \,,
\]
whence Theorem \ref{1.1} implies that $G \in \{C_2^{r-1} \oplus C_4, C_2 \oplus C_{2n}\}$ with $n, r \ge 2$. Suppose now that $n, r \ge 3$. Clearly, Condition (b) in Proposition \ref{3.5} is equivalent to
\begin{itemize}
\item[(b')] $\{2, \mathsf D (G)-1, \mathsf D (G)\} \in \mathcal L (G)$.
\end{itemize}
Thus Proposition \ref{3.5} implies in particular that $\mathcal L ( C_2 \oplus C_{2n} ) \ne \mathcal L ( C_2^{r-1} \oplus C_4 )$, and thus the assertion of Corollary \ref{1.3} follows.
\end{proof}

\providecommand{\bysame}{\leavevmode\hbox to3em{\hrulefill}\thinspace}
\providecommand{\MR}{\relax\ifhmode\unskip\space\fi MR }
\providecommand{\MRhref}[2]{%
  \href{http://www.ams.org/mathscinet-getitem?mr=#1}{#2}
}
\providecommand{\href}[2]{#2}

\end{document}